\documentclass[10pt]{amsart}
\usepackage{amsmath}
\usepackage{amssymb}
\usepackage{amsfonts}
\usepackage{amsthm}
\usepackage{enumerate}
\usepackage[all]{xy}
\usepackage[latin1]{inputenc}
\usepackage{mathdots}
\usepackage{graphicx}
\usepackage{latexsym}
\usepackage{mathrsfs}
\PassOptionsToPackage{usenames,dvipsnames,svgnames}{xcolor}
\usepackage{tikz}

\usetikzlibrary{shapes,arrows,positioning,automata}
\definecolor{cof}{RGB}{219,144,71}
\definecolor{pur}{RGB}{186,146,162}
\definecolor{greeo}{RGB}{91,173,69}
\definecolor{greet}{RGB}{52,111,72}

\input xy

\newtheorem{Theo}{Theorem}[section]
\newtheorem{Prop}[Theo]{Proposition}
\newtheorem{Cor}[Theo]{Corollary}
\newtheorem{Lemma}[Theo]{Lemma}

\theoremstyle{definition}

\newtheorem{Exam}[Theo]{Example}

\def\mystrut(#1,#2){\vrule height #1pt depth #2pt width 0pt}
\newcommand{\rep}{{\rm rep}}

\newcommand{\Hom}{{\rm Hom}}
\newcommand{\Ext}{{\rm Ext}}

\newcommand{\Z}{\mathbb{Z}}
\newcommand{\C}{\mathcal{C}}
\newcommand{\A}{\mathcal{A}}

\newcommand{\ch}{\check}

\newcommand{\ext}{{\rm ext}}
\newcommand{\sperp}{\rotatebox[origin=c]{90}{$\models$}}

\begin{document}

\author{Charles Paquette}
\address{Charles Paquette, Department of Mathematics, University of Connecticut, Storrs, CT, 06269-3009, USA.}
\email{charles.paquette@uconn.edu}

\title{Accumulation points of real Schur roots}

\maketitle

\begin{abstract} Let $k$ be an algebraically closed field and $Q$ be an acyclic quiver with $n$ vertices. Consider the category $\rep(Q)$ of finite dimensional representations of $Q$ over $k$. The exceptional representations of $Q$, that is, the indecomposable objects of $\rep(Q)$ without self-extensions, correspond to the so-called real Schur roots of the usual root system attached to $Q$. These roots are special elements of the Grothendieck group $\Z^n$ of $\rep(Q)$. When we identify the dimension vectors of the representations (that is, the non-negative vectors of $\Z^n$) up to positive multiple, we see that the real Schur roots can accumulate in some directions of $\mathbb{R}^n \supset \Z^n$. This paper is devoted to the study of these accumulation points. After giving new properties of the canonical decomposition of dimension vectors, we show how to use this decomposition to describe the rational accumulation points. Finally, we study the irrational accumulation points and we give a complete description of them in case $Q$ is of weakly hyperbolic type.
\end{abstract}

\section*{Introduction}

One important problem in representation theory of algebras is to describe the non-isomorphic indecomposable modules of a given finite dimensional algebra. Let $A$ be an associative finite dimensional $k$-algebra, where $k$ is an algebraically closed field. It has been shown by Drozd \cite{Drozd} that the category of finitely generated right $A$-modules is either tame or wild, but not both. In the tame case, one can hope to get a classification of the non-isomorphic indecomposable modules, since for a fixed dimension vector $d$, the indecomposable modules of dimension vector $d$, up to isomorphism, is a union of a finite set with finitely many one-parameter families. In the wild case, it is usually hopeless, though many things could still be said. It is known from Kac's results \cite{Kac0, Kac} that when $A$ is hereditary, thus Morita equivalent to a path algebra $kQ$ where $Q$ is an acyclic quiver, the dimension vectors of the indecomposable modules (or representations) are the so-called positive roots of the root system associated to the Kac-Moody Lie algebra of the underlying graph of $Q$. Therefore, it is an important problem to understand these positive roots. In particular, one should try to have a better understanding of the real roots: they form a nice subset of the set of all positive roots and the representation theory attached to these roots is better understood. In \cite{Vivien1, Vivien2, Vivien3}, the authors have studied accumulation points of real roots in more general root systems, where accumulation has a meaning when one looks at the rays associated to the roots. In a root system defined from an acyclic quiver, it is straightforward to check that the real roots always accumulate to the hypersurface defined by the Tits form of the quiver. Some fractal-like properties of these accumulation points are uncovered in \cite{Vivien1}.

\medskip

Geometric representation theory of algebras provides nice tools for studying the positive roots of a quiver $Q$. In geometric representation theory, one often restricts to (positive) Schur roots, which have nicer geometric properties.
The real Schur roots then correspond to the so-called exceptional representations, which are the indecomposable representations without self-extensions. These representations play a fundamental role in tilting theory, perpendicular categories and in the theory of cluster algebras.
In this paper, we propose to study accumulation points of real Schur roots of an acyclic quiver $Q$. We will see how the problem is related to the study of canonical decomposition of a dimension vector. In Section one, which is an introductory section, we will recall the definition of roots, Schur roots,  the Euler-Ringel form and the Coxeter transformation. We will give the precise definition of an accumulation point of dimension vectors. In Section two, we will consider quivers of weakly hyperbolic type, which are quivers of wild type such that the Tits form of $Q$ has nice convex properties. We will provide nice properties of these quivers and show how to build a wide family of such quivers. In Section three, we will consider the so-called canonical decomposition of a dimension vector and the closely related notions of Schur sequences and exceptional sequences. In Section four, we will give a new property of canonical decompositions related to imaginary Schur roots. Finally, in Section five, we will study accumulation points of real Schur roots in general. Our main theorem states that an isotropic Schur root is a rational accumulation point of real Schur roots. Conversely, a rational accumulation point of real Schur roots has canonical decomposition involving only isotropic Schur roots. We will give a complete description of all accumulation points (rational and irrational) of real Schur roots for quivers of weakly hyperbolic type. They could all be obtained using accumulation points of quivers with two vertices.

\section{Euler form, roots and Schur roots}

Throughout this paper, $Q=(Q_0,Q_1)$ denotes a connected acyclic quiver with $n = |Q_0|$ vertices and $k$ denotes an algebraically closed field. For an arrow $\alpha \in Q_1$, we denote by $t(\alpha) \in Q_0$ its \emph{tail} and by $h(\alpha) \in Q_0$ its \emph{head}. For convenience, we will assume that $Q_0 = \{1,2, \ldots, n\}$. We denote by $\rep(Q)$ the category of finite dimensional representations of $Q$, which is a Hom-finite hereditary abelian category equivalent to the category of finite dimensional right $kQ$-modules. For each $x \in Q_0$, we denote by $S_x, P_x$ and $I_x$ the simple representation at $x$, the indecomposable projective representation at $x$ and the indecomposable injective representation at $x$, respectively. Finally, for $d=(d_1, \ldots,d_n) \in (\Z_{\ge 0})^n$, we denote by $\rep(Q,d)$ the set of all representations $M$ such that for $i \in Q_0$, $M(i) = k^{d_i}$. Recall that in geometric representation theory, $\rep(Q,d)$ is identified with $\prod_{\alpha \in Q_1}M_{d_{h(\alpha)} \times d_{t(\alpha)}}(k)$, where $M_{i \times j}(k)$ is the set of all $i \times j$ matrices over $k$. The algebraic group $${\rm GL}_d(k):=\prod_{1 \le i \le n}{\rm GL}_{d_i}(k)$$ acts on $\rep(Q,d)$ in the following way. For $g=(g_1, \ldots, g_n) \in {\rm GL}_d(k)$ and $M=(M_{\alpha})_{\alpha \in Q_1}\in \rep(Q,d)$, we have $$g \cdot M = \left(g_{h(\alpha)}M_{\alpha}g_{t(\alpha)}^{-1}\right)_{\alpha \in Q_1}.$$

\medskip

\subsection{Euler form}

Given a representation $M$ of $Q$, we denote by $d_M$ its dimension vector, which is an element in the Grothendieck group
$K_0(\rep(Q)) = \mathbb{Z}^n$ of $\rep(Q)$, and where $(d_M)_i$, for $1 \le i \le n$, is the $k$-dimension of $M(i)$. Let $\langle -,- \rangle$ stand for the bilinear form defined on $\mathbb{Z}^n$ as follows.
If $d=(d_1,\ldots,d_n) \in \mathbb{Z}^n$ and $e=(e_1,\ldots,e_n) \in \mathbb{Z}^n$, then
$$\langle d,e \rangle = \sum_{i=1}^n{d_ie_i} \; - \hspace{-5pt} \sum_{\alpha \in Q_1} d_{t(\alpha)}e_{h(\alpha)}.$$
This is known as the \emph{Euler-Ringel form} associated to $Q$. This is a (non-symmetric) bilinear form defined on the Grothendieck group of $\rep(Q)$. We denote by $E_Q$ (or simply $E$ when there is no risk of confusion) the $n \times n$ matrix of this Euler-Ringel form for $Q$ in the canonical basis, that is, the basis given by the dimension vectors of the simple representations. Recall that if $M,N$ are two representations of $Q$, then $$\langle d_M, d_N \rangle = d_M^T E d_N = {\rm dim}_k \Hom(M,N) - {\rm dim}_k \Ext^1(M,N).$$ In the sequel, we set $A_Q = E_Q + E_Q^T$ (or simply $A$ when there is no risk of confusion) for the matrix of the symmetrized form $(-,-)_Q$ corresponding to the Euler-Ringel form.  Moreover, for $x \in \mathbb{R}^n$, we write $q(x)$ for $\langle x,x \rangle = \frac{1}{2}(x,x)$ which is known as the \emph{Tits form of $Q$} or the \emph{quadratic form associated to $Q$}. The homogeneous equation $q(x)=0$ is referred to as the \emph{quadric associated to $Q$}.

\medskip

\subsection{Roots}

A vector $v=(v_1,\ldots,v_n)$ in $\mathbb{R}^n$ is \emph{non-negative} if $v_i \ge 0$ for all $i$ and \emph{positive} if it is non-negative and non-zero. It is called \emph{strictly positive} provided $v_i > 0$ for all $i$. Similarly, it is \emph{non-positive} if $v_i \le 0$ for all $i$ and \emph{negative} if it is non-positive and non-zero. It is called \emph{strictly negative} provided  $v_i < 0$ for all $i$. Given a vector $d \in \Z^n$, write supp$(d)$ for the support of $d$, that is, the subset of vertices $i$ in $Q_0$ with $d_i \ne 0$. For $1 \le i \le n$, set $e_i:= d_{S_i}$ and let $\mathfrak{S}_i$ be the $n \times n$ matrix such that $\mathfrak{S}_i\cdot d = d - ( d, e_i ) e_i$. The matrix $\mathfrak{S}_i$ is called the \emph{reflection} at $i$. The \emph{Weyl group} $\mathcal{W}_Q$ of $Q$ is the multiplicative group generated by all the matrices $\mathfrak{S}_1, \ldots, \mathfrak{S}_n$.  A \emph{real root} of $Q$ is any $d \in \mathbb{Z}^n$ such that $w d = e_i$ for some $w \in \mathcal{W}_Q$ and some $1 \le i \le n$. A nice fact about real roots is that a real root is either a positive vector or a negative vector. Hence we get the corresponding notions of a \emph{positive real root} and a \emph{negative real root}. Moreover, a negative real root is the negative of a positive real root. Now, set $$W = \{d \; \text{positive vector in}\; \Z^n \mid ( d, e_i ) \le 0 \; \text{for all} \; i, {\rm supp}(d) \; \text{is connected}\}.$$
A \emph{positive imaginary root} is any element of the form $wd$ for some $d \in W$ and $w \in \mathcal{W}_Q$. A positive imaginary root is always a positive vector. A \emph{negative imaginary root} is just the negative of a positive imaginary root. Observe that an imaginary root $d$ satisfies $q(d) \le 0$. Such a root will be called \emph{isotropic} if $q(d) = 0$ and \emph{strictly imaginary} if $q(d) < 0$. Now, if we denote by $\Phi$ the set of all roots of $Q$, we get
$$\Phi = \Phi_{\rm Re} \cup \Phi_{\rm Im} = (\Phi^+_{\rm Re} \cup -\Phi^+_{\rm Re}) \cup (\Phi^+_{\rm Im} \cup -\Phi^+_{\rm Im}),$$
where $\Phi_{\rm Re}$ is the set of all real roots, $\Phi^+_{\rm Re}$ the set of all positive real roots, $\Phi_{\rm Im}$ the set of all imaginary roots and $\Phi^+_{\rm Im}$ the set of all positive imaginary roots. All unions are disjoint unions.

\medskip

From Kac's results, if $M$ is an indecomposable representation, then $d_M$ is a positive root. Conversely, if $d$ is a positive root, then $d = d_M$ for some indecomposable representation $M$. In geometric representation theory, one often restricts to Schur roots, that is, the dimension vectors of the representations having a trivial endomorphism ring. Such representations are called \emph{Schur representations}. It is well known that a positive root $d$ is Schur if and only if the variety $\rep(Q,d)$ of $d$-dimensional representations of $Q$ contains a dense open set whose objects are all indecomposable. The \emph{real Schur roots} are the positive real roots that are Schur roots and the \emph{imaginary Schur roots} are the positive imaginary roots that are Schur roots. If $d$ is a positive real root, then there is a unique indecomposable representation $M$, up to isomorphism, such that $d = d_M$. If, in addition, $d$ is Schur, then $\Ext^1(M,M)=0$. Hence the real Schur roots correspond bijectively to the \emph{exceptional representations}, up to isomorphism.

\medskip

\subsection{Definition of $\Delta_Q$ and accumulation points}

Two nonzero vectors $v_1, v_2$ in $\mathbb{R}^n$ are said to be \emph{equivalent} if one is a positive multiple of the other. We denote by $\Delta_Q$ the positive cone in $\mathbb{R}^n$ modulo the equivalence relation above. We can think of $\Delta_Q$ as the set of all rays in the positive orthant of $\mathbb{R}^n$. If $v \in \mathbb{R}^n$ is positive, then the corresponding point in $\Delta_Q$ is denoted by $[v]$, or simply by $v$ when there is no risk of confusion. A ray $r$ in $\Delta_Q$ is \emph{rational} if there exists a (nonzero) rational point in $r$. Observe that the nonzero dimension vectors modulo the above equivalence relation correspond precisely to the rational rays of $\Delta_Q$.

\medskip

Now, we need to give a precise definition of an accumulation point in $\Delta_Q$. Given a positive vector $\alpha \in \mathbb{R}^n$, let $s(\alpha)$ be the sum of the entries of $\alpha$. We define $\check{\alpha}$ to be $(1/s(\alpha))\alpha$ and we call it the \emph{normalized vector associated to $\alpha$}. If $\alpha$ is a root, then $\ch \alpha$ is a \emph{normalized root}. Observe that $\alpha, \ch \alpha$ lie in the same ray in $\Delta_Q$. We denote by $\Delta(1)$ the subset of $\mathbb{R}^n$ consisting of the positive normalized vectors. Every point in $\Delta(1)$ can be written uniquely as a convex combination of the simple dimension vectors. Also, every ray $\alpha$ in $\Delta_Q$ is uniquely represented by the element $\check{\alpha}$ in $\Delta(1)$. Observe that if $\alpha$ is a root, then $\check{\alpha}$ need not be a root, unless it is a simple root.  Of course, $\langle \alpha, \alpha \rangle$ and $\langle \check{\alpha}, \check{\alpha} \rangle$ have the same sign. Since $\Delta(1)$ is a topological space (with the canonical Euclidean topology inherited from $\mathbb{R}^n$) and every ray in $\Delta_Q$ has a unique representative in $\Delta(1)$, this gives $\Delta_Q$ the structure of a topological space. An open set in $\Delta_Q$ is then just the cone over an open set in $\Delta(1)$. A sequence $(r_n)_{n \ge 0}$ of rays in $\Delta_Q$ converges if the sequence $(\ch r_n)_{n \ge 0}$ of vectors converges in $\Delta(1)$. If $d_n \in r_n$ for all $n \ge 0$, we will say that the sequence of vectors $(d_n)_{n \ge 0}$ \emph{converges}. In this case, when the sequence $(r_n)_{n \ge 0}$ contains an infinite number of pairwise distinct rays, the corresponding limit ray is then called an \emph{accumulation point} or \emph{accumulation ray}. When this accumulation point is a rational ray, it will often be identified with a dimension vector in it, or with the corresponding element in $\Delta(1)$.

\medskip

\subsection{Coxeter transformation}

We denote by $C_Q$ (or simply $C$ when there is no risk of confusion) the Coxeter matrix for $Q$, that is, $C = -E^{-1}E^{T}$. It is well known that $C \in \mathcal{W}_Q$. Recall that if $M$ is indecomposable and non-projective in $\rep(Q)$, then $C\cdot d_M$ is the dimension vector of the Auslander-Reiten translate $\tau M$ of $M$. If $M = P_x$, then $C \cdot d_{P_x}  = -d_{I_x}$. Similarly, if $M$ is indecomposable and non-injective in $\rep(Q)$, then $C^{-1}\cdot d_M$ is the dimension vector of $\tau^{-1} M$. If $M = I_x$, then $C^{-1} \cdot d_{I_x}  = -d_{P_x}$.

\medskip

The following discussion about eigenvectors of $C$ can be found in \cite{Pena}. If $Q$ is of Dynkin type, then the only possibility for a real eigenvalue of $C$ is $\lambda = -1$ (and it may not happen).  If $Q$ is of Euclidean type, then $\lambda = -1$ \emph{may} be an eigenvalue of $C$ and $\lambda = 1$ is an eigenvalue of $C$ of algebraic multiplicity $2$ and geometric multiplicity $1$. There are no other real eigenvalues. If $Q$ is of wild type, then there is a largest positive real eigenvalue $\lambda_+ > 1$ with geometric multiplicity $1$ and a smallest positive real eigenvalue $\lambda_- = 1/\lambda_+$ also of geometric multiplicity $1$. Moreover, there exist strictly positive eigenvectors $y_+,y_-$ associated to the eigenvalues $\lambda_+$ and $\lambda_-$, respectively. The corresponding points in $\Delta_Q$ are called the \emph{special eigenvectors} for $Q$, and the corresponding eigenvalues are called the \emph{special eigenvalues}. In case $Q$ is of Euclidean type, the null root $\delta$ (that is, the only isotropic Schur root) is a strictly positive vector and an eigenvector corresponding to $\lambda =1$.  The corresponding point in $\Delta_Q$ will also be called the \emph{special eigenvector} of $Q$, and $\lambda = 1$ the \emph{special eigenvalue}. In this case, we set $\delta = y^- = y^+$. Therefore, when $Q$ is connected of infinite type, we have one or two special eigenvectors in $\Delta_Q$.

\section{Weakly hyperbolic quivers}

A quiver $Q$ is said to be \emph{weakly hyperbolic} or of \emph{weakly hyperbolic type} if the matrix $A_Q$ has exactly one negative eigenvalue and the others are positive. More generally, it is called \emph{at most weakly hyperbolic} if it has at most one non-positive eigenvalue. Hence a quiver of Dynkin or Euclidean (extended Dynkin) type is at most weakly hyperbolic. From \cite[Lemma 4.5]{KacBook}, if $Q$ is connected and $A_Q$ contains exactly one zero eigenvalue and all the others are positive, then $Q$ is of Euclidean type; and if all eigenvalues of $A_Q$ are positive, then $A$ is of Dynkin type. Hence, a connected quiver $Q$ which is at most weakly hyperbolic is either of Dynkin type, of Euclidean type, or is weakly hyperbolic. Note that the notion of being weakly hyperbolic is more general then that of being hyperbolic: a connected quiver $Q$ which is not of Dynkin or Euclidean type is said to be \emph{hyperbolic} if every full subquiver is a union of quivers of Euclidean and Dynkin types. The reader is referred to \cite[page 50]{KacBook} for more details on quivers of hyperbolic type. Note that a quiver of hyperbolic type is connected by definition. However, a quiver which is (at most) weakly hyperbolic needs not be connected.

\medskip

The following lemma will be useful in the sequel. In essentially follows from the description of the imaginary cone by Kac.

\begin{Lemma} \label{boundary1}
Let $c_1,c_2$ be two accumulation points of real roots in $\Delta_Q$.  Then the cone spanned by the rays $c_1,c_2$ passes through the region $q(x) \le 0$.  It lies on the boundary $q(x)=0$ if and only if $(c_1,c_2) =0$, that is, if and only if $\langle c_1, c_2 \rangle = - \langle c_2, c_1 \rangle$.
\end{Lemma}

\begin{proof}
By \cite[Lemma 5.8]{KacBook}, we know that the closure of the imaginary cone contains the convex hull of the accumulation points of the real roots. Hence, the line segment joining $c_1$ and $c_2$ is in the closure of the imaginary cone. Now, any point $x$ in the imaginary cone satisfies $q(x) \le 0$ by \cite[Exercice 5.10e]{KacBook}. Hence, the same holds for the closure of the imaginary cone.
\end{proof}

\begin{Lemma} \label{twoOnes}Let $Q$ be at most weakly hyperbolic. If $x,y$ are positive vectors and lie on the quadric $q(z)=0$, then the line segment joining $x,y$ lies in the region $q(z) \le 0$.
\end{Lemma}

\begin{proof} We may assume that the matrix $A = A_Q$ has exactly one negative eigenvalue, all the other eigenvalues being positive. Then it is easily seen that the region $q(z) \le 0$ is the cone over an higher dimensional ellipse. Let $C_1, C_2=-C_1$ be the two convex components of that cone. Let $P$ be the positive orthant in $\mathbb{R}^n$, which is seen as (half of) a cone. If one of $P \cap C_1, P\cap C_2$ is zero, then the result is clear. So assume both intersections are nonzero. Then $C_i \ne P \cap C_i$ for $i = 1,2$. Therefore, each $C_i$ intersects the boundary of $P$. Since a boundary of $P$ corresponds to a proper subquiver of $Q$, we know that for $i=1,2$, $P\cap C_i$ contains a point $c_i$ which is an accumulation point of positive real roots. By Lemma \ref{boundary1}, the line segment joining $c_1,c_2$ lies in $q(z) \le 0$, hence lies in one of $P \cap C_1, P \cap C_2$, which is impossible.
\end{proof}

The following result characterizes quivers which are at most weakly hyperbolic.

\begin{Prop} \label{LemmaPlane}
The quiver $Q$ is at most weakly hyperbolic if and only if the quadric $q(z)=0$ does not contain a plane.
\end{Prop}

\begin{proof}
Choose a basis in $\mathbb{R}^n$ such that the matrix $D$ of the symmetrized form $(-,-)_Q$ is diagonal with $0,1,-1$ as entries. By Sylvester's Law of Inertia, $Q$ is at most weakly hyperbolic if and only if $D$ contains at least $n-1$ diagonal entries which are ones. Suppose that $Q$ is at most weakly hyperbolic. If $A$ has no negative eigenvalue, then $A$ is positive or semi-positive definite, and it is then well known that $Q$ is of Dynkin or Euclidean type. The quadric is either zero or a ray, and the result is clear.
Suppose that $A$ has one negative eigenvalue, hence exactly one diagonal entry of $D$ is a $-1$.
Hence we may assume that for $x = (x_1, \ldots, x_n)$, the equation $q(x)=0$ is given by
$$x_1^2 = x_2^2 + \cdots + x_n^2.$$
Assume to the contrary that the quadric contains a plane. Let $u,v \in \mathbb{R}^n$ be the direction vectors of a plane in the quadric. Since $q(u+v)=q(u)=q(v)=0$, we get
$$(u_1 + v_1)^2 = (u_2 + v_2)^2 + \cdots + (u_n + v_n)^2.$$
Therefore,
$$u_1v_1 = u_2v_2 + \cdots + u_nv_n.$$
Denote by $\langle -,- \rangle'$ the scalar product in $\mathbb{R}^{n-1}$ such that for $x=(x_2, \ldots, x_n), y=(y_2,\ldots, y_n)$, $\langle x, y \rangle' = x_2y_2 + \cdots + x_ny_n$. Let $u' = (u_2,\ldots, u_n)$ and $v' = (v_2, \ldots, v_n)$. We get
$$\langle u', u' \rangle' \langle v', v' \rangle' = u_1^2v_1^2 = (\langle u', v' \rangle')^2.$$
By the Cauchy-Schwarz inequality, $u', v'$ are linearly dependent. There exists a real number $\lambda$ with $u' = \lambda v'$. Considering $q(u - \lambda v) = 0$, we get $(u_1 - \lambda v_1)^2=0$ and hence, $u_1 = \lambda v_1$.  Therefore, $u,v$ are linearly dependent, a contradiction. This proves the necessity.

Assume now that the quadric does not contain a plane. Assume to the contrary that $Q$ is not at most weakly hyperbolic. Hence, $D$ has at least two diagonal entries which are not positive.  If $D$ has at least two diagonal entries which are zeros, then it is clear that $z^TDz = 0$ contains a plane, and so is $q(z)=0$. Hence, at most one diagonal entry of $D$ is zero. With no loss of generality, we may assume that the first two diagonal entries $d_{11}, d_{22}$ of $D$ are not positive, and $d_{11}$ is the only diagonal entry which is possibly zero. Thus, $d_{22}=-1$. On the other hand, since $A_Q$ is neither negative definite nor negative semi-definite, there exists a diagonal entry in $D$ which is a one. Assume that the third diagonal entry $d_{33}$ is one. For $1 \le i \le n$, let $e_i \in \mathbb{R}^n$ be the vector with a one as the $i$-th entry and zeros everywhere else. Assume first that $d_{11}=0$. Then the vectors $e_1, e_2+e_3$ are the direction vectors of a plane in $z^TDz=0$, a contradiction. So assume $d_{11}=-1$. We claim that at least two diagonal entries in $D$ are ones. Assume otherwise, so all the diagonal entries but $d_{33}$ are $-1$. As argued in the proof of Lemma \ref{twoOnes}, the region $q(z) \ge 0$ is a cone over an higher dimensional ellipse. Let $C_1, C_2=-C_1$ be the two convex components of that cone. Let $P$ be the positive orthant in $\mathbb{R}^n$. All the boundary rays of $P$ lie in $q(z) \ge 0$. If all of $P$ lies in $q(z) \ge 0$, then $Q$ is of Dynkin or Euclidean type, and we get a contradiction. Let $p_1, \ldots, p_n$ be the boundary rays of $P$. These boundary rays can be partitioned into two non-empty sets $S_1, S_2$ with $S_1 \subseteq C_1$, $S_2 \subseteq C_2$. We may assume that we have three distinct rays $p,q,r$ such that $p \in S_1$ and $q,r \in S_2$. The line segment joining $p$ and $q$ then contains elements in $q(z) \le 0$, and hence corresponds to a wild subquiver of $Q$ having two vertices. There exists an accumulation point $c_1$ of real roots lying on this line segment and in $C_1$. Similarly, there exists an accumulation point $c_2$ of real roots lying on the line segment joining $p$ and $r$ and in $C_1$. By Lemma \ref{boundary1}, the line segment joining $c_1, c_2$ is in $q(z) \le 0$. But since $C_1$ is convex, this line is also in $q(z) \ge 0$. Thus, the line is on the boundary $q(z)=0$. But then, the plane with direction vectors $c_1,c_2$ lies on the boundary $q(z)=0$, a contradiction. This proves the claim. So we may assume that the third and fourth diagonal entries $d_{33}, d_{44}$ of $D$ are ones.  Thus, we see that the vectors $e_1 + e_3, e_2 + e_4$ are the direction vectors of a plane in $z^TDz=0$, a contradiction.
\end{proof}

The following characterizes the quivers of weakly hyperbolic type.

\begin{Prop} \label{PropWeakly}
The quiver $Q$ is of weakly hyperbolic type if and only if ${\rm det}(A_Q) < 0$ and there exists a full subquiver $Q'$ with $n-1$ vertices which is at most weakly hyperbolic.
\end{Prop}

\begin{proof}
Suppose that $Q$ is of weakly hyperbolic type. Clearly, ${\rm det}(A_Q) < 0$. Let $Q'$ be any full subquiver with $n-1$ vertices. By Proposition \ref{LemmaPlane}, the quadric of $Q$ does not contain a plane. Therefore, the same holds for the quadric of $Q'$. Hence, by the use of Proposition \ref{LemmaPlane} again, $Q'$ is at most weakly hyperbolic. This proves the necessity. For the sufficiency, assume that ${\rm det}(A_Q) < 0$ and there exists a full subquiver $Q'$ with $n-1$ vertices which is at most weakly hyperbolic. We may assume that the upper left principal submatrix of $A$ of size $(n-1) \times (n-1)$ is $A':=A_{Q'}$. Therefore, there exists an invertible matrix $S'$ such that $S'A'S'^T = {\rm diag}(I_{n-2}, a)$ where $a \in \{-1, 0 ,1\}$. Let $S = {\rm diag}(S',1)$. Then $SAS^T$ is a symmetric matrix whose upper left principal submatrix of size $(n-1) \times (n-1)$ is $S'A'S'^T = {\rm diag}(I_{n-2}, a)$. Now, it is easy to see that there exists an invertible matrix $T$ such that $TSAS^TT^T = {\rm diag}(I_{n-2},B)$ for some $2 \times 2$ matrix $B$. Now, there exists an invertible matrix $W$ such that $WTSAS^TT^TW^T = {\rm diag}(I_{n-2},c,d)$. Since ${\rm det}(A_Q) < 0$, we may assume that $c > 0$ and $d < 0$. By Sylvester's Law of Inertia, the number of negative (or positive) eigenvalues of $A$ is the number of negative (resp. positive) diagonal elements in ${\rm diag}(I_{n-2},c,d)$. Hence $A$ has exactly one negative eigenvalue and $n-1$ positive eigenvalues.
\end{proof}

The following proposition gives a wide family of quivers that are of weakly hyperbolic type.

\begin{Prop}
The following type of acyclic quivers are at most weakly hyperbolic.
\begin{enumerate}[$(a)$]
    \item A (connected) quiver of hyperbolic type: every proper subquiver of it is a union of quivers of Dynkin or Euclidean types.
\item A quiver with $3$ vertices.
\item A quiver $Q$ for which ${\rm det}(A_Q) < 0$ and there exists a full subquiver with $n-1$ vertices which is a union of quivers of Dynkin type and at most one quiver of Euclidean type.
\item A quiver for which ${\rm det}(A_Q) < 0$ and there exists a vertex $x$ whose arrows are all connected to a given vertex $y$, and the full subquiver generated by all the vertices but $x,y$ is a union of quivers of Dynkin type and at most one quiver of Euclidean type.
\end{enumerate}
\end{Prop}

\begin{proof}
For the first statement, see \cite[Exercice 4.6, page 51]{KacBook}. Suppose that $Q$ is connected and has three vertices. If det$(A) < 0$, then we are done since $A=A_Q$ cannot have three negative eigenvalues. Therefore, suppose that det$(A) \ge 0$. Write
$$A=\left(%
\begin{array}{ccc}
  2 & -a & -b \\
  -a & 2 & -c \\
  -b & -c & 2 \\
\end{array}%
\right)$$
where $a,b,c$ are non-negative integers.  We have det$(A) = 2(4 - abc - a^2-b^2-c^2)$.  Since $Q$ is connected, notice that det$(A)=0$ if and only if $Q$ is of type $\widetilde{\mathbb{A}}$. So suppose det$(A)>0$. Then $0 \le a,b,c \le 1$ and at least one of them is zero. This implies that $Q$ is of Dynkin type.

For the remaining statements, observe that a quiver which is a union of quivers of Dynkin type and at most one quiver of Euclidean type is at most weakly hyperbolic. In statement (d), the full subquiver generated by all vertices but $y$ is of weakly hyperbolic type. Hence, these statements follow from Proposition \ref{PropWeakly}.
\end{proof}

\begin{Exam}
(1) Here is an example of a quiver with $4$ vertices which is not at most weakly hyperbolic.  The matrix $A_Q$ has two positive and two negative eigenvalues. However, every full subquiver is at most weakly hyperbolic.
$$\xymatrix{\circ \ar@/^/[r] \ar@/_/[r] \ar[r] & \circ \ar[r] & \circ \ar@/^/[r] \ar@/_/[r] \ar[r] & \circ}$$
(2) Here is an example such as part (d) of the above proposition. Hence, this quiver is weakly hyperbolic.
$$
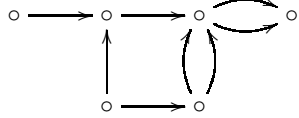
\xymatrix{\circ \ar[r] & \circ \ar[r] & \circ \ar@/_/[r] \ar@/^/[r] & \circ\\ & \circ \ar[u] \ar[r] & \circ \ar@/_/[u] \ar@/^/[u] & }$$
\end{Exam}

\section{The canonical decomposition and Schur sequences}

In this section, $Q$ is a connected acyclic quiver whose vertices are denoted $1,2,\ldots,n$ and we may assume that $i> j$ whenever there is an arrow $i \to j$. Due to results of Kac \cite{Kac}, if $d$ is a dimension vector, then there is a decomposition, denoted $d=\alpha_1 \oplus \cdots \oplus \alpha_m$, having the property that there exists an open dense subset $\mathcal{U}_d$ of $\rep(Q,d)$ such that for $M \in \mathcal{U}_d$, we have $M \cong M_1 \oplus \cdots \oplus M_m$, where each $M_{i}$ is a Schur representation with $d_{M_{i}} = \alpha_i$. Moreover, $\Ext^1(M_i,M_j)=0$ when $i \ne j$. The latter decomposition of $d$ is unique up to ordering, and is called the \emph{canonical decomposition} of $d$. The dimension vectors $\alpha_i$ are clearly Schur roots, however, they do not need be distinct. Sometimes, it is more convenient to write the above decomposition as
$$(*) \quad d=p_1\beta_1 \oplus \cdots \oplus p_r\beta_r,$$
where the $\beta_i$ are pairwise distinct and $p_i$ is the number of $1 \le j \le m$ with $\beta_i = \alpha_j$. It follows from \cite[Theorem 3.8]{SchofieldGeneral} that when $\beta_i$ is strictly imaginary, then $p_i=1$. When writing a canonical decomposition as in $(*)$, we adopt the convention that when $\alpha$ is a strictly imaginary Schur root and $p$ is a positive integer, then $p\alpha$ is just one root (not $p$ times the root $\alpha$ as when $\alpha$ is real or isotropic). With this convention, Schofield has proven in \cite[Theorem 3.8]{SchofieldGeneral} the following result.

\begin{Prop}[Schofield]\label{CanDecompSchofield} Let $d=p_1\beta_1 \oplus \cdots \oplus p_r\beta_r$ be the canonical decomposition of $d$. If $p$ is a positive integer, then $pd=pp_1\beta_1 \oplus \cdots \oplus pp_r\beta_r$ is the canonical decomposition of $pd$, using the above convention for strictly imaginary Schur roots.
\end{Prop}

There is another equivalent way to describe the canonical decomposition of a dimension vector. For two dimension vectors $d_1,d_2$, let $\ext(d_1,d_2)$ denote the minimal value of ${\rm dim}_k \Ext^1(M_1,M_2)$ where $(M_1,M_2) \in \rep(Q,d_1) \times \rep(Q,d_2)$. Similarly, let $\hom(d_1,d_2)$ denote the minimal value of ${\rm dim}_k \Hom(M_1,M_2)$ where $(M_1,M_2) \in \rep(Q,d_1) \times \rep(Q,d_2)$. There are open dense subsets $\mathcal{U}_1 \subseteq \rep(Q,d_1)$, $\mathcal{U}_2 \subseteq \rep(Q,d_2)$ such that for $(M_1,M_2) \in \mathcal{U}_1 \times \mathcal{U}_2$, we have $$\hom(d_1,d_2) = {\rm dim}_k\Hom(M_1,M_2) \;\; \text{and} \;\;\ext(d_1,d_2) = {\rm dim}_k\Ext^1(M_1,M_2).$$ In particular, for $(M_1,M_2) \in \mathcal{U}_1 \times \mathcal{U}_2$, we have $$\langle d_{M_1}, d_{M_2} \rangle = \langle d_1, d_2 \rangle = \hom(d_1,d_2) - \ext(d_1,d_2).$$ The following is due to Kac; see \cite{Kac}.

\begin{Prop}[Kac] Let $d = d_1 + \cdots + d_m$ be a decomposition of a dimension vector $d$ as a sum of dimension vectors. This is the canonical decomposition of $d$ if and only if each $d_i$ is a Schur root and $\ext(d_i,d_j)=0$ for $i \ne j$.
\end{Prop}

It is possible to find the canonical decomposition of a dimension vector in practice. Schofield, in \cite{SchofieldGeneral}, gave an inductive way to find this canonical decomposition. Later, in \cite{DWAlgo}, Derksen and Weyman gave an elementary algorithm to find the canonical decomposition. The validity of their algorithm gave a new structural property of the canonical decomposition of a dimension vector. Given two dimension vectors $d_1,d_2$, we say that $d_1$ is \emph{left orthogonal} to $d_2$, written $d_1 \perp d_2$, if $\hom(d_1,d_2) = 0 = \ext(d_1,d_2)$. In an expression as in $(*)$, Derksen and Weyman have shown that it is possible to order the $\beta_i$ in such a way that for $i < j$, $\beta_i \perp \beta_j$.

\medskip

In this section, we will use Derksen-Weyman's algorithm to give other property of canonical decompositions. We first need to introduce more terminology on sequences of Schur roots.
A \emph{weak Schur sequence} $(\alpha_1, \ldots, \alpha_r)$ is any sequence of Schur roots where $\alpha_i \perp \alpha_j$ whenever $i < j$. Now, if $\alpha, \beta$ are Schur roots with $\alpha \perp \beta$, then it is well known that the weight space SI$(Q,\beta)_{\langle \alpha, - \rangle}$ of the ring of semi-invariants SI$(Q,\beta)$ is nonzero. Following \cite{DWSchurSeqn}, we write $\alpha \circ \beta$ for the dimension over $k$ of SI$(Q,\beta)_{\langle \alpha, - \rangle}$. Hence, $\alpha \perp \beta$ implies $\alpha \circ \beta > 0$. We say that $\alpha$ is \emph{strongly left orthogonal} to $\beta$, written $\alpha \sperp \beta$, if $\alpha \circ \beta =  1$. A weak Schur sequence $(\alpha_1, \alpha_2, \ldots, \alpha_r)$ such that $\alpha_i \sperp \alpha_j$ whenever $i < j$ is called a \emph{Schur sequence}. Not all weak Schur sequences are Schur sequences. For instance, if $\delta$ is an isotropic Schur root, then $(\delta,\delta)$ is a weak Schur sequence but is not a Schur sequence. The following lemmas appear in \cite{DWSchurSeqn}.

\begin{Lemma}[Derksen-Weyman] \label{ConstructionSchurSequences}Let $\mathcal{S} = (\alpha_1, \alpha_2, \ldots, \alpha_r)$ be a weak Schur sequence where all $\alpha_i$ are real Schur roots.
Then $\mathcal{S}$ is a Schur sequence.
\end{Lemma}

\begin{Lemma}[Derksen-Weyman] \label{ConstructionSchurSequences2}Let $\mathcal{S} = (\alpha_1, \alpha_2, \ldots, \alpha_r)$ be a weak Schur sequence such that $(\alpha_2, \ldots, \alpha_r)$ is a Schur sequence. If $\alpha_1$ is real, then $S$ is a Schur sequence.
\end{Lemma}

\begin{Lemma}[Derksen-Weyman] \label{ConstructionSchurSequences3}Let $\mathcal{S} = (\alpha_1, \alpha_2, \ldots, \alpha_r)$ be a Schur sequence. Then $\alpha_1, \alpha_2, \ldots, \alpha_r$ are linearly independent.
\end{Lemma}

Given a dimension vector $d$ and a Schur sequence $\mathcal{S}=(\alpha_1, \ldots, \alpha_r)$, we say that $\mathcal{S}$ is a Schur sequence \emph{for} $d$ if $d$ can be written as a non-negative linear combination of the roots in $\mathcal{S}$. In this case, we can write $d = p_1 \alpha_1 + \cdots + p_r\alpha_r$ and we define $c(\mathcal{S}, d) := (p_1, p_2, \ldots, p_r)\in (\mathbb{Z}_{\ge 0})^r$. Sometimes, we will be interested in the case where all $p_i$ are positive, that is, when $c(\mathcal{S}, d)$ is strictly positive.

\medskip

If we are given a dimension vector $d = (d_1, \ldots, d_n)$, then it follows from our hypothesis on the labeling of $Q_0$ that
$\mathcal{S}=(S_1, S_2, \ldots, S_n)$ is a weak Schur sequence for $d$ with $c(\mathcal{S},d)=(d_1, d_2, \ldots, d_n)$. It is also a Schur sequence because all simple roots are real. The latter is called the \emph{trivial Schur sequence} for $d$. Of course, for a given dimension vector, there are many possible Schur sequences for it.
A Schur sequence $(\alpha_1, \ldots, \alpha_r)$ such that $\langle \alpha_j, \alpha_i \rangle \ge 0$ whenever $i < j$ is said to be \emph{final}. Note that as a consequence of Derksen-Weyman's algorithm, one gets the following result.

\begin{Prop}[Derksen-Weyman] If $\mathcal{S}=(\alpha_1, \alpha_2, \ldots, \alpha_r)$ is a final Schur sequence for $d$ with $c(\mathcal{S},d)=(p_1, p_2, \ldots, p_r)$ strictly positive, then
$$p_1\alpha_1 \oplus \cdots \oplus p_r\alpha_r$$
is the canonical decomposition of $d$. Conversely, given the canonical decomposition $$q_1\beta_1 \oplus \cdots \oplus q_s\beta_s$$ of $d$ with the $\beta_i$ pairwise distinct, we may order the $\beta_i$ in such a way that $\mathcal{T}=(\beta_1, \beta_2, \ldots, \beta_s)$ is a final Schur sequence for $d$.
\end{Prop}

The following two lemmas and corollary are about quivers of weakly hyperbolic type. We say that a vector $x$ lies \emph{inside (or outside, or on)} the quadric of $Q$ if $q(x) < 0$ (resp. $q(x) > 0$, or $q(x)=0$).

\begin{Lemma} \label{TrivialLemma}
Let $Q$ be at most weakly hyperbolic and let $x$ be on the quadric.  Then $\{y \mid (x,y)_Q=0, q(y)=0\}$ are the multiples of $x$.
\end{Lemma}

\begin{proof}
Assume otherwise.  Then there exists a point $y$, which is not a multiple of $x$, such that $q(y)=0$ and $(x,y)_Q=0$.  This means that the plane generated by $x$ and $y$ lies on the quadric $q(x)=0$, which is contrary to Lemma \ref{LemmaPlane}.
\end{proof}

\begin{Lemma} \label{LemmaWeaklyHyper}
Let $Q$ be of weakly hyperbolic type and let $x$ be a nonzero point inside the quadric.  Then any nonzero point $y$ with $(x,y)=0$ does not lie inside the quadric.
\end{Lemma}

\begin{proof}
Let $y$ be a nonzero point with $(x,y)=0$. Suppose to the contrary that $y$ lies inside the quadric.  Recall that the symmetric matrix $A_Q$ has a negative eigenvalue and all the others are positive. The result is easy if $n=2$: we have that $x^TA_Q$ is negative and hence $(x,y) \ne 0$ if $y \ne 0$. Assume $n \ge 3$. We are given two nonzero points $x,y$ with $q(x)<0$ and $q(y)<0$ such that $(x,y)=0$.  Consider the hyperplane $H_x$ of equation $(x,-)=0$. We have a set $\{v_1, \ldots, v_{n-1}\}$ of $n-1$ orthogonal eigenvectors of $A_Q$ corresponding to positive eigenvalues. Let $H$ be the hyperplane generated by all these $v_i$. We see that for $0 \ne z \in H$, we have $q(z)>0$. Since $H \ne H_x$, the intersection $H \cap H_x$ is of dimension $n-2 > 0$. We claim that there exists a nonzero point $y_0$ in $H_x$ which lies on the quadric. Assume otherwise.  Since $y \in H_x$ and $q$ is continuous, every nonzero point $u\in H_x$ (and in $H \cap H_x$) is such that $q(u) < 0$. Hence $n=2$, a contradiction. Thus, there exists a nonzero point $y_0$ with $(x,y_0)=0$ and $q(y_0)=0$.  Let $H_{y_0}$ be the hyperplane of equation $(-,y_0)=0$.  By the dual of Lemma \ref{TrivialLemma}, the quadratic form restricted to $H_{y_0}$ is negative semi-definite (since $H_{y_0}$ contains $x$). Since $H_{y_0}$ and  $H$ are distinct, the dimension of the intersection $H_{y_0} \cap H$ is $n-2$ and $H_{y_0} \cap H$ lies on the quadric. Hence, $H_{y_0} \cap H = 0$, thus $n=2$, a contradiction.
\end{proof}

The following corollary is a consequence of Lemma \ref{LemmaWeaklyHyper} and \cite[Theorem 4.1]{SchofieldGeneral}. Its says, in particular, that in the canonical decomposition of a dimension vector supported by a quiver which is at most weakly hyperbolic, at most one imaginary Schur root may occur, up to multiplicity.

\begin{Cor}
Let $Q$ be at most weakly hyperbolic. If $\mu,\nu$ are two imaginary roots appearing in the same Schur sequence, then $(\mu, \nu) < 0$.  In particular, one of $\langle \mu,\nu \rangle$, $\langle \nu,\mu \rangle$ is negative.
\end{Cor}

In the sequel, Derksen-Weyman's algorithm will be used, and the reader is referred to \cite{DWAlgo} for a description of this algorithm. For convenience, we denote by (DW) this algorithm. We use the same notation and numbering for the steps of (DW): they are numbered from $(1)$ to $(6)$.
Let us just briefly say what the algorithm does by giving the general lines, without details.
Let $d$ be a dimension vector. The algorithm starts with any Schur sequence $\mathcal{S}$ for $d$. At every step, the coefficients of $d$ corresponding to a Schur sequence are non-negative. A \emph{loop} of (DW) just changes a Schur sequence for $d$ to another Schur sequence for $d$, which may have smaller length. It is the use of steps (5) and (6) once. After each loop (so going from a Schur sequence $\mathcal{T}_1$ for $d$ to another Schur sequence $\mathcal{T}_2$ for $d$), the sum of the coefficients in $c(\mathcal{T}_2,d)$ is smaller than the sum of the coefficients in $c(\mathcal{T}_1,d)$. The algorithm terminates when the Schur sequence is final. Note that, after a loop, we may have a root with a zero coefficient attached to it. This root is discarded in (DW), so that the length of the Schur sequence decreases by one. Moreover, for a given loop, two Schur roots may be combined to get one imaginary Schur root, which will also have as effect to decrease the length of the Schur sequence by one. At any time, the length of a sequence is bounded by $n$. More precisely, the number of real Schur roots plus twice the number of imaginary Schur roots is bounded by $n$.

\medskip

\begin{Lemma} \label{LemmaStrictlyInE}
Let $d$ be a dimension vector and suppose that we have a Schur sequence $E$ such that $d$ is a positive linear combination of the roots in $E$. If $E$ contains a strictly imaginary Schur root, then the canonical decomposition of $d$ contains a strictly imaginary Schur root.
\end{Lemma}

\begin{proof}
Apply the algorithm (DW) for $d$ starting with $E$. By a careful check at steps (6)(b)(ii), (6)(c)(ii) and (6)(d), we see that after each loop, the  Schur sequence produced by the algorithm contains a strictly imaginary Schur root, and the coefficient attached to it is positive. In particular, the canonical decomposition of $d$ contains a strictly imaginary Schur root.
\end{proof}

Given a Schur sequence of real Schur roots $\mathcal{S}$, we denote by $\C(\mathcal{S})$ the thick subcategory generated by the objects of the corresponding exceptional sequence. It is well known that $\C(\mathcal{S})$ is equivalent to the category of representations of an acyclic quiver having the same number of vertices as the length of $\mathcal{S}$. The following lemma explains how to use the refinement theorem in \cite{DWSchurSeqn} in the context of a Schur sequence having no strictly imaginary Schur roots.

\begin{Lemma} \label{refinement}
Let $(\alpha_1, \alpha_2, \ldots, \alpha_r)$ be a Schur sequence for $d$ without strictly imaginary Schur roots and suppose that $\alpha_i$ is isotropic for some $i$ with $i$ minimal. Then there exist two real Schur roots $\beta, \gamma$ such that $(\alpha_1, \alpha_2, \ldots,\alpha_{i-1}, \beta, \gamma, \alpha_{i+1}, \ldots, \alpha_r)$ is a Schur sequence with $\alpha_i = \beta + \gamma$.
\end{Lemma}

\begin{proof} We proceed by induction on $n+s(\alpha_i)$, where $n$ is the number of vertices of the quiver and $s(\alpha_i)$ is the sum of the coefficients in $\alpha_i$. If $n=2$, then $r=1$ and $Q$ is the Kronecker quiver. In this case, we take $\beta, \gamma$ the simple roots of $Q$ such that $(\beta, \gamma)$ is a Schur sequence. Assume that $n > 2$. Of course, $s(\alpha_i) \ge 2$.  If $\alpha_1$ is real, then the Schur sequence $(\alpha_2, \ldots, \alpha_r)$ can be seen as a Schur sequence of an acyclic quiver with $n-1$ vertices, and we can use induction. So we may assume that $\alpha_1$ is isotropic, that is, $i=1$. Consider $\sigma = -\langle -, \delta \rangle$, where $\delta = \sum_{i=2}^r\alpha_i$. As argued in \cite[¸proof of the refinement theorem]{DWSchurSeqn}, $\alpha_1$ is not $\sigma$-stable and hence, we have a non-trivial $\sigma$-stable decomposition of $\alpha_1$:
$$\alpha_1 = p_1\beta_1 + \cdots +p_s\beta_s,$$
where $s > 1$, and for $1 \le i \le s,$ the $p_i$ are positive integers and $\beta_i$ are $\sigma$-stable Schur roots. The $\beta_i$ can be ordered so that $\mathcal{T}=(\beta_1, \ldots, \beta_s)$ is a Schur sequence. In particular, we see that $\mathcal{T}$ is a Schur sequence for $\alpha_1$. As argued in \cite[proof of the refinement theorem]{DWSchurSeqn}, the sequence
$$\mathcal{W} = (\beta_1, \ldots, \beta_s, \alpha_2, \ldots, \alpha_r)$$
is a Schur sequence, and hence is a Schur sequence for $d$. We may start applying loops of the algorithm (DW) starting with the Schur sequence $\mathcal{W}$ for $d$ and only care about the subsequence $\mathcal{T}$. Since $\alpha_1$ is isotropic, by Lemma \ref{LemmaStrictlyInE}, none of the $\beta_i$ is strictly imaginary. Also, since the roots in $\mathcal{T}$ are linearly independent by Lemma \ref{ConstructionSchurSequences3}, at most one $\beta_i$ is isotropic. After a certain number of loops of (DW), we will have a Schur sequence
$$(\mu, \nu, \alpha_2, \ldots, \alpha_r)$$
for $d$ where at most one of $\mu, \nu$ is isotropic and $(\mu, \nu)$ is a Schur sequence for $\alpha_1$. If both $\mu, \nu$ are real, then consider $\beta, \gamma$ the simple roots in $\C(\mu, \nu)$. The sequence $(\beta, \gamma, \alpha_2, \ldots, \alpha_r)$ is Schur by Lemma \ref{ConstructionSchurSequences2} and is a desired sequence. Otherwise, one of $\mu, \nu$ is isotropic. If $\mu$ is real, we may use induction for $(\nu, \alpha_2, \ldots, \alpha_r)$ since $n-1 < n$ and $s(\nu) < s(\alpha_1)$, and we get a Schur sequence $(\beta', \gamma', \alpha_2, \ldots, \alpha_r)$ where $\nu= \beta'+ \gamma'$ and $\beta', \gamma'$ are real. Since $\mu$ is real, by Lemma \ref{ConstructionSchurSequences2}, we get that $(\mu, \beta', \gamma', \alpha_2, \ldots, \alpha_r)$ is a Schur sequence. Since the Schur sequence $(\mu, \nu)$ is not final, we have that $\langle \nu, \mu \rangle < 0$ and $\alpha_1 = \nu - \langle \nu, \mu \rangle \mu$. Taking the Schur sequence $(\mu, \beta', \gamma')$ and reflecting the roots $\beta', \gamma'$ to the left of $\mu$ yields a Schur sequence of real Schur roots $(\epsilon_1, \epsilon_2, \mu)$, where $\epsilon_1 = \pm (\beta' - \langle \beta', \mu \rangle \mu), \epsilon_2 = \pm (\gamma' - \langle \gamma', \mu \rangle \mu)$. We see now that $\alpha_1$ is a root in $\C(\epsilon_1, \epsilon_2)$. Take $\phi_1, \phi_2$ be the simple roots in $\C(\epsilon_1, \epsilon_2)$ with $(\phi_1, \phi_2)$ a Schur sequence. By Lemma \ref{ConstructionSchurSequences2}, the sequence $(\phi_1, \phi_2, \mu, \alpha_2, \ldots, \alpha_r)$ is Schur. Now, $\alpha_1 = \phi_1+ \phi_2$ and the latter sequence is a desired sequence.
Suppose now that $\mu$ is isotropic. Then $s(\mu) < s(\alpha_1)$ and by induction, there is a Schur sequence $(\beta', \gamma', \nu, \alpha_2, \ldots, \alpha_r)$ where $\mu= \beta'+ \gamma'$ and $\beta', \gamma'$ are real. Since the Schur sequence $(\mu, \nu)$ is not final, we have that $\langle \nu, \mu \rangle < 0$ and $\alpha_1 = \mu - \langle \nu, \mu \rangle \nu$. Taking the Schur sequence $(\beta', \gamma', \nu)$ and reflecting the roots $\beta', \gamma'$ to the right of $\nu$ yields a Schur sequence of real Schur roots $(\nu, \epsilon_1, \epsilon_2)$, where $\epsilon_1 = \pm (\beta' - \langle \nu, \beta' \rangle \nu), \epsilon_2 = \pm (\gamma' - \langle \nu, \gamma' \rangle \nu)$. We see now that $\alpha_1$ is a root in $\C(\epsilon_1, \epsilon_2)$. Take $\phi_1, \phi_2$ be the simple roots in $\C(\epsilon_1, \epsilon_2)$ with $(\phi_1, \phi_2)$ a Schur sequence. Finally, take $\phi_3$ a real Schur root in $\C(\mu, \phi_1, \phi_2)$ such that $(\phi_1, \phi_2, \phi_3)$ is a Schur sequence. By Lemma \ref{ConstructionSchurSequences2}, the sequence $(\phi_1, \phi_2, \phi_3, \alpha_2, \ldots, \alpha_r)$ is Schur.  This is a desired Schur sequence.
\end{proof}

The following lemma will be very handy for completing Schur sequences on the "right side".

\begin{Lemma} \label{CompleteNonProjective}
Let $(M_{r+1},M_{r+2},\ldots,M_n)$ with $1 \le r \le n-1$ be an exceptional sequence in $\rep(Q)$ such that $\oplus_i M_i$ is sincere. Then there exists an exceptional sequence $E:=(M_r, M_{r+1}, \ldots, M_n)$ such that $M_r$ is not projective in $\C(E)$.
\end{Lemma}

\begin{proof}If $r = 1$, then $M_1$ is unique and it is not projective since otherwise, $\oplus_{i = r+1}^n M_i$ would not be sincere. Consider the category $\C_1 = \C(M_{r+1}, \ldots, M_n)$ and $\C_2 = ^\perp \C_1$. Since $\C_2$ is generated by an exceptional sequence of length $r$, $\C_2$ is indeed generated by its simple objects $M_1, \ldots, M_r$. These simple objects may be ordered to form an exceptional sequence and hence, we may assume that $(M_1, \ldots, M_r, M_{r+1}, \ldots, M_n)$ is exceptional and $\C_2 = \C(M_1, \ldots, M_r)$. We claim that the given $M_r$ satisfies the property of the statement. Set $d_i = d_{M_i}$. Since the $M_i$, $1 \le i \le r$, are the simple objects in $\C_2$, we have $m_{i,j}:=-\langle d_j, d_i \rangle \ge 0$ for $1 \le i<j \le r$. By using successive reflections, we get exceptional objects $N_i$, $1 \le i \le r-1$ such that we have an exceptional sequence $(M_1, \ldots, M_{i-1}, N_i, M_{i}, \ldots, M_{r-1})$ for $2 \le i \le r-1$ and an exceptional sequence $(N_1,M_1, \ldots,M_{r-1})$. Observe that $d_{N_i} = d_r + \sum_{i \le j \le r-1}r_jd_j$ where $r_j$ is the sum of all the $2^{r-j-1}$ products of the form $m_{r, i_1} m_{i_1, i_2} \cdots m_{i_{s-1}, i_s}m_{i_s,j}$ where $j < i_s < \cdots < i_1 < r$. Now, $N_1$ is a projective object in $\C_2$. By hypothesis, it is not projective in $\rep(Q)$. Considering the exceptional sequence
$$(N_1, M_1, M_2, \ldots, M_{r-1}, M_{r+1}, \ldots, M_n),$$
and by reflecting $N_1$ to the right, we get an exceptional sequence
$$(M_1, M_2, \ldots, M_{r-1}, M_{r+1}, \ldots, M_n, Z)$$
and where $Z$ has to coincide with $\tau N_1$. Observe also that $Z \in \C(E)$, since we can get $Z$ by reflecting $M_r$ to the right of $M_{r+1}, \ldots, M_n$. Now, $N_1, Z$ lie on opposite sides of the hyperplane $H$ in $\Delta_Q$ generated by $d_1, \ldots, d_{r-1}, d_{r+1}, \ldots, d_n$. If $M_r$ was projective in $\C(E)$, then $d_r,d_Z$ would lie on the same side of the hyperplane $H'$ generated by $d_{r+1}, \ldots, d_n$. Suppose that $H$ is given by a homogeneous equation $f(-)=0$. Then $f(d_{N_1}), f(d_Z)$ have opposite sign. However, we see that $f(d_{N_1}) = f(d_r)$. Hence, $d_r,d_Z$ lie on opposite sides of $H$. Since $H'$ is a subspace of $H$, we see that $d_r,d_Z$ lie on opposite sides of $H'$ and this proves the claim.
\end{proof}

\section{Strictly imaginary Schur roots}

In this section, we want to study strictly imaginary Schur roots appearing in canonical decompositions. We will show that if a dimension vector $d$ contains a strictly imaginary Schur root in its canonical decomposition, then there is a small neighborhood $\mathcal{U}$ of $d$ such that the dimension vectors in $\mathcal{U}$ satisfy the same property. The reader is referred to Section $1$ for a description of the topology on $\Delta_Q$.

\medskip

We first start with some new definitions that will simplify the notations. Given a set of vectors $\alpha_1, \ldots, \alpha_r$ in $\mathbb{R}^n$, denote by $H(\alpha_1, \ldots, \alpha_r)$ the subspace generated by these vectors. For convenience, let $H(\alpha_1, \ldots, \alpha_r)_{\ge 0}$ (or $H(\alpha_1, \ldots, \alpha_r)_+$) denote the subset of $H(\alpha_1, \ldots, \alpha_r)$ of those vectors that can be written as a non-negative (resp. positive) linear combination of $\alpha_1, \ldots, \alpha_r$. If $E=(\alpha_1, \ldots, \alpha_r)$ is a Schur sequence, then we will simply write $H(E)$ for $H(\alpha_1, \ldots, \alpha_r)$.

\medskip

If $E=(\alpha_1, \alpha_2, \ldots, \alpha_r)$ is a Schur sequence, then we may consider the corresponding \emph{normalized Schur sequence} $\ch E = (\check{\alpha_1}, \ldots, \check{\alpha_r})$. The following lemma gives the relationship between Schur sequences and normalized Schur sequences. Recall that for a nonzero dimension vector $d$, $s(d)$ denotes the sum of the entries of $d$.

\begin{Lemma} \label{NormalizedToNot}
Fix $d$ a dimension vector and $E=(\alpha_1, \ldots, \alpha_r)$ a Schur sequence with $d \in H(E)$. Write $d = \sum_{i=1}^rp_i\alpha_i$. Then $\check{d} \in H(\ch E)$ with $\ch d = \sum_{i=1}^r\frac{p_i s(\alpha_i)}{s(d)}\ch \alpha_i$.
\end{Lemma}

\medskip

Now, let $E=(\alpha_1, \ldots, \alpha_r)$ be a Schur sequence. Suppose that we have a dimension vector $d$ with $d \in H(E)$. Then the expression
$d = \sum_{i=1}^rp_i\alpha_i$ of $d$ as a linear combination of the $\alpha_i$ becomes
$$\ch d = \frac{p_1 s(\alpha_1)}{s(d)}\ch\alpha_1 + \cdots + \frac{p_r s(\alpha_r)}{s(d)}\ch\alpha_r$$
in $\Delta(1)$. Consider the function $$f(E,\alpha_i): \Delta_Q \to [0,1]$$ which associates to an arbitrary ray $[d]$ of $\Delta_Q$ the coefficient $\frac{p_i s(\alpha_i)}{s(d)}$ as above. Then one can check that $f(E,\alpha_i)$ is well defined and is continuous.
Before proving the main theorem of this section, we need the following easy lemma.

\begin{Lemma} \label{LemmaInsideQuadric}
Let $d$ be a dimension vector with $\langle d,d \rangle < 0$. Then $d$ has a strictly imaginary root in its canonical decomposition.
\end{Lemma}

\begin{proof}
Assume otherwise. Using the algorithm (DW) in \cite{DWAlgo}, there exists a final Schur sequence $E=(\alpha_1, \ldots, \alpha_r)$ such that $d \in H(E)_+$ and none of the $\alpha_i$ are strictly imaginary. Write $d = \sum_{i=1}^rp_i \alpha_i$. Since $\langle \alpha_i, \alpha_j \rangle\ge 0$ for all $1 \le i,j \le r$, we get $\langle d,d \rangle \ge 0$, a contradiction.
\end{proof}

We are now ready to prove the main result of this section.

\begin{Theo} \label{MainTheorem}
Suppose that a strictly imaginary root appears in the canonical decomposition of a dimension vector $d$. Then there is a small neighborhood $\mathcal{U}$ of $d$ such that the canonical decomposition of any dimension vector in $\mathcal{U}$ also involves a strictly imaginary Schur root.
\end{Theo}

\begin{proof}
Apply the algorithm (DW) to $d$, until it yields a Schur sequence of the form $$(\alpha_1, \ldots, \alpha_{j}, \alpha_{j+1}, \ldots, \alpha_r),$$ where none of the roots are strictly imaginary, $\langle \alpha_{j+1}, \alpha_j \rangle < 0$ and $p_j\alpha_j + p_{j+1}\alpha_{j+1}$ is strictly imaginary where $d = p_1\alpha_1 + \cdots + p_r\alpha_r$. By Lemma \ref{refinement}, consider a Schur sequence  $$\mathcal{T} = (\beta_1, \ldots, \beta_s),$$
of real Schur roots where each isotropic Schur root $\alpha_i$ is replaced by two real Schur roots. The subsequence $(\alpha_j, \alpha_{j+1})$ then corresponds to a full subsequence $\mathcal{S}$ of length $2,3$ or $4$ in $\mathcal{T}$.

We have two cases to consider depending on whether $d$ is sincere or not. Suppose first that $d$ is sincere. Using the dual of Lemma \ref{CompleteNonProjective}, and using the fact that a Schur sequence of real Schur roots corresponds to an exceptional sequence, we deduce that there is a Schur sequence of real Schur roots $(\beta_1, \beta_2, \ldots, \beta_s, \ldots, \beta_n)$ such that for $1 \le i \le n-s$, $\beta_{s+i}$ is not the dimension vector of an injective object in $\C(\beta_1, \ldots, \beta_{s+i})$. Observe that any dimension vector $f$ can be written as a linear combination of the $\beta_i$.

\smallskip

We claim that for any $f \in \Z^n$, there exists a Schur sequence of real Schur roots $E_f = (\mu_1, \ldots, \mu_q, \beta_1, \ldots, \beta_s, \gamma_1, \ldots, \gamma_t)$ of length $n$ such that when writing $f$ as a linear combination of the roots of $E_f$, the coefficients in front of the $\mu_i, \gamma_j$ are all non-negative. For $0 \le j \le n-s$, denote by $\C_j$ the category $\C(\beta_1, \ldots, \beta_s, \ldots, \beta_{s+j})$ and by $G_j$ its Grothendieck group, seen as a subgroup of $\Z^n$. More generally, to prove the claim, we will prove by induction on $0 \le j \le n-s$, that for any  $f \in G_j$, there exists a Schur sequence of real Schur roots $E_{j,f}$ of shape $(\mu_{1}, \ldots, \mu_{q}, \beta_1, \ldots, \beta_s, \gamma_{1}, \ldots, \gamma_{t})$ of length $s+j$ in $\C_j$ such that when writing $f$ as a linear combination of the roots of $E_{j,f}$, the coefficients in front of the $\mu_{i}, \gamma_{j}$ are all non-negative. The claim is the particular case where $j = n-s$.
If $j =0$, then the claim holds trivially. So assume that $j > 0$. Let $f \in G_j$ and write $f  = q_{1}\beta_1+ \cdots + q_{s+j}\beta_{s+j}$.
If $q_{s+j} \ge 0$, then $f' = f - q_{s+j}\beta_{s+j}$ is a vector in $G_{j-1}$. By induction, there is a Schur sequence of real Schur roots $E_{j-1,f'} = (\mu_{1}, \ldots, \mu_{q}, \beta_1, \ldots, \beta_s, \gamma_{1}, \ldots, \gamma_{t})$ of length $s+j-1$ in $\C_{j-1}$ such that when writing $f'$ as a linear combination of the roots of $E_{j-1,f'}$, the coefficients in front of the $\mu_{i}, \gamma_{j}$ are all non-negative. Then we set $E_{j,f} = (\mu_{1}, \ldots, \mu_{q}, \beta_1, \ldots, \beta_s, \gamma_{1}, \ldots, \gamma_{t}, \beta_{s+j})$. So assume that $q_{s+j} < 0$. Since $\beta_{s+j}$ is not the dimension vector of an injective object in $\C_j$, $\beta_{s+j}$ and $\tau^{-1}_{\C_j}\beta_{s+j}$ lie on opposite sides of the hyperplane $H(\beta_1, \ldots, \beta_{s+j-1})$.
Therefore, in the expression of $f$ as a linear combination of $\tau^{-1}_{\C_j}\beta_{s+j}, \beta_1, \ldots, \beta_{s+j-1}$, the coefficient $q$ of $\tau^{-1}_{\C_j}\beta_{s+j}$ will be positive. Consider $f' = f - q\tau^{-1}_{\C_j}\beta_{s+j}$ which is a vector in $G_{j-1}$. By induction, there is a Schur sequence of real Schur roots $E_{j-1,f'} = (\mu_1, \ldots, \mu_q, \beta_1, \ldots, \beta_s, \gamma_1, \ldots, \gamma_t)$ of length $s+j-1$ such that when writing $f'$ as a linear combination of the roots of $E_{j-,f'}$, the coefficients in front of the $\mu_i, \gamma_j$ are all non-negative. Then we set $E_{j,f} = (\tau^{-1}_{\C_j}\beta_{s+j}, \mu_1, \ldots, \mu_q, \beta_1, \ldots, \beta_s, \gamma_1, \ldots, \gamma_t)$. This proves our claim.

\smallskip

The proof of the claim also yields a finite set $\{E_1, E_2, \ldots, E_g\}$ of Schur sequences of real Schur roots such that for each $1 \le i \le g$, $E_i$ contains $(\beta_1, \ldots, \beta_s)$ as a full subsequence. Moreover, for each $f \in \Z^n$, there exists $i$ such that we may take $E_f = E_i$. For each $1 \le i \le g$, let $F_i$ be the set of elements $f \in \Z^n$ such that we may take $E_f = E_i$. Of course, $d \in F_i$ for all $1 \le i \le g$ and $\Z^n = \cup_{i=1}^gF_i$. Now, fix $i$ with $1 \le i \le g$ and consider $f \in F_i$. Consider the coefficients $a_1, \ldots a_q, b_1, \ldots, b_s, c_1, \ldots, c_t$ and $a_1', \ldots a_q', b_1', \ldots, b_s', c_1', \ldots, c_t'$ when writing $d, f$ as a linear combination of the roots in $E_i = (\mu_1, \ldots, \mu_q, \beta_1, \ldots, \beta_s, \gamma_1, \ldots, \gamma_t)$, respectively. Consider the full subsequence $\mathcal{S} = (\beta_p, \ldots, \beta_{p+a})$ of length $a=2,3$ or $a=4$ in $(\beta_1, \ldots, \beta_s)$ as constructed above. Observe that $$\sum_{i=1}^ap_{p+i}\beta_{p+i}$$ is strictly imaginary and hence, $$q\left(\sum_{i=1}^ap_{p+i}\beta_{p+i}\right) < 0.$$ There exists a small neighborhood $U_i$ of $d$ in $\Delta_Q$ such that for $f \in F_i \cap U_i$, $$p_1', \ldots, p_s' > 0 \;\; \text{and} \; \; q\left(\sum_{i=1}^ap_{p+i}'\beta_{p+i}\right) < 0.$$ By Lemma \ref{LemmaInsideQuadric}, a strictly imaginary root appears in the canonical decomposition of $\sum_{i=1}^ap_{p+i}'\beta_{p+i}$, and hence, the same happens for $f$. Since $F_i$ is closed under positive integer scaling, we set $\mathcal{U} = \cup_{1 \le i \le g}{\rm Cone}_{\mathbb{R}}(F_i) \cap U_i$ and we see that any dimension vector $f \in \mathcal{U}$ has a strictly imaginary Schur root in its canonical decomposition.

\smallskip

Consider now the case where $d$ is not sincere. Let $E$ be the support of $d$. Let $E^c = Q_0 \backslash E$. For $i \in E^c$, take $V_i$ be the largest submodule of the injective $I_i$ such that the support of $V_i$ is included in $E \cup \{i\}$. Let $\A$ be the thick subcategory of $\rep(Q)$ generated by the simple representations $S_j$ for $j \in E$. We claim that the $V_i$, $i \in E^c$, are the relative simple objects of $\A^\perp$. We first need to show that $V_i \in \A^\perp$ for $i \in E^c$. Fix $i \in E^c$ and $M \in \A$. Consider a morphism $u: M \to V_i$. Since $\A$ is closed under quotients, the image of $u$ lies in $\A$. If $u \ne 0$, then the image of $u$ has $S_i$ as a composition factor, which is a contradiction. Hence, $\Hom(M,V_i)=0$. If $V_i = I_i$, then $\Ext^1(M,V_i)=0$. Otherwise, observe that if $R_0 \to R_1$ is a minimal injective co-presentation of $V_i$, then $R_0 = I_i$ and $R_1$ has a socle whose support lies in $E^c$. In particular, the top of $\tau^{-1}V_i$ has a support in $E^c$. This gives $0 = \Hom(\tau^{-1}V_i, M) \cong \Ext^1(M,V_i)$. This proves that $V_i$ lies in $\A^\perp$. Now, any proper quotient of $V_i$ lies in $\A$ and hence do not lie in $\A^\perp$. This proves our claim. Let $n' = |E|$ and $s' = n - n'$. As in the first part of the proof, there exists a Schur sequence of real Schur roots  $(\beta_1, \beta_2, \ldots, \beta_s, \ldots, \beta_{n'})$ such that for $1 \le i \le n'-s$, $\beta_{s+i}$ is not the dimension vector of an injective object in $\C(\beta_1, \ldots, \beta_{s+i})$. Now, we may order the $V_i$ to get an exceptional sequence $(V_{i_1}, \ldots, V_{i_{s'}})$. Set $\beta_{n'+j}$ be the dimension vector of $V_{i_j}$. Observe that when writing a dimension vector
using the $\beta_i$, the coefficients in front of $\beta_{n'+1}, \ldots, \beta_{n}$ are non-negative. This easily follows from the description of the $V_i$. Now, we can use the argument of the case where $d$ is sincere.
\end{proof}

\section{Accumulation points}

In this section, $Q$ is an acyclic quiver. We study the accumulation points of real Schur roots in $\Delta_Q$. In the first part of this section, we study accumulation points in general. In the second part, we give a complete description of these accumulation points when $Q$ is of weakly hyperbolic type. Recall from Section $1$ that $\Delta_Q$ has a natural topology which is the one inherited from the Euclidean topology on $\Delta(1)$. Recall also that a sequence $(\alpha_i)_{i \ge 1}$ of dimension vectors accumulates in $\Delta_Q$ if the $\alpha_i$ are pairwise distinct as elements in $\Delta_Q$ and the sequence $(\ch \alpha_i)_{i \ge 1}$ converges in $\Delta(1)$.

\medskip

We first need the following easy well known fact. We give a proof for the sake of completeness.

\begin{Prop} \label{HLP}
If $r$ is an accumulation point of real roots, then $\ch r$ lies on the quadric of $Q$, that is, $q(\ch r) = \langle \ch r, \ch r \rangle = 0$.
\end{Prop}

\begin{proof}
Let $(\alpha_i)_{i \ge 1}$ be a sequence of real roots which is an accumulation point. Thus, we may assume that the $\alpha_i$ are pairwise distinct. We have
$$q(\ch \alpha_i) = \left\langle \frac{\alpha_i}{s(\alpha_i)}, \frac{\alpha_i}{s(\alpha_i)}\right\rangle = \frac{\langle \alpha_i, \alpha_i \rangle}{s^2(\alpha_i)} = \frac{1}{s^2(\alpha_i)}.$$
Hence, we need to show that $(s(\alpha_i))_{i \ge 1}$ is not bounded. If it is bounded, then there is a dimension vector $d$ such that $\alpha_i = d$ for infinitely many $i$, a contradiction.
\end{proof}

In the following lemma, a \emph{preprojective root} is any root of the form $C^{-i}d_{P_j}$ for some $i \ge 0$ and $1 \le j \le n$; and a \emph{preinjective root} is any root of the form $C^{i}d_{I_j}$ for some $i \ge 0$ and $1 \le j \le n$.

\begin{Lemma} \label{Lemmay+}
Let $(\alpha_i)_{i \ge 1}$ be a sequence of pairwise distinct preprojective (or preinjective) roots. Then the sequence converges to $y^-$ ($y^+$, respectively).
\end{Lemma}

\begin{proof} We only prove the preprojective case, the other being dual. Here $Q$ is necessarily of Euclidean type or of wild type. Suppose first that $Q$ is of wild type. Recall (or see \cite{Campo, Ringel}) that the spectral radius of $C^{-1}$ is $\rho(C^{-1}) = \lambda_+ > 1$ and is an eigenvalue of $C^{-1}$ with geometric multiplicity one. Moreover, there is a corresponding eigenvector $y^-$ which is strictly positive. Let $X$ be a preprojective representation. By \cite[Theorem in 3.5]{Pena}, we have
$$\lim_{m \to \infty} \frac{1}{\rho(C^{-1})^m} d_{\tau^{-m} X} = \lambda_X y^-,$$
where $\lambda_X > 0$. Hence in $\Delta_Q$, the sequence $(\alpha_1,\alpha_2,\ldots)$ converges to $y^-$.
Suppose now that $Q$ is Euclidean.  Then the null root $\delta$ is an eigenvector of $C^{-1}$ whose corresponding eigenvalue is one.  We claim that the sequence converges to $\delta$. From Proposition \ref{HLP}, any infinite sequence $(\beta_1, \beta_2, \ldots)$ of pairwise distinct real Schur roots accumulates to a ray on the quadric $\langle x, x \rangle = 0$, which contains only one ray, namely the ray of $\delta$.  This proves the last case.
\end{proof}

By the above lemma, the special eigenvectors $y^-, y^+$ are accumulation points of real Schur roots. The following result can be found in \cite[Theo. 3.5]{Pena}.

\begin{Lemma} \label{LemmaGrowthTau}
Let $X$ be an exceptional representation. If $X$ is not preinjective, then the sequence $(d_{\tau^{-i}X})_{i \ge 0}$ converges to $y^-$ in $\Delta_Q$. If $X$ is not preprojective, then the sequence $(d_{\tau^{i}X})_{i \ge 0}$ converges to $y^+$ in $\Delta_Q$.
\end{Lemma}

The following proposition is essentially a consequence of Lemma \ref{refinement}.

\begin{Prop} \label{PropAccuPoints}If $d$ is an isotropic Schur root, then $d$ an isotropic Schur root of a finitely generated rank two tame subcategory of $\rep(Q)$.
\end{Prop}

We now have the following theorem, which gives a description of the rational accumulation points of real Schur roots.

\begin{Theo} \label{TheoAccuPoints} If $d$ is an isotropic Schur root, then $d$ is a rational accumulation point of real Schur roots. Conversely, if $d$ is a rational accumulation point of real Schur roots, then the canonical decomposition of $d$ only involves pairwise orthogonal isotropic Schur roots.
\end{Theo}

\begin{proof}The first part follows from Proposition \ref{PropAccuPoints} and the fact that the result holds for the Kronecker quiver. For the second part, assume that $d$ is a dimension vector that is an accumulation point of real Schur root. It follows from Proposition \ref{HLP} that $\langle d, d \rangle =0$. Let $(\alpha_1, \ldots, \alpha_n)$ be a Schur sequence which is the output of (DW) applied to $d$. Then the $\alpha_i$ are linearly independent Schur roots with $\langle \alpha_i, \alpha_j \rangle = 0$ and  $\langle \alpha_j, \alpha_i \rangle \ge 0$ whenever $i < j$. Since $d$ is an accumulation point, it follows from Theorem \ref{MainTheorem} that for each $i$, $\alpha_i$ is real or isotropic. In particular, $\langle \alpha_i, \alpha_i \rangle \ge 0$ for all $i$. Now, we have $d = \sum_{i=1}^rp_i\alpha_i$ with $p_i > 0$ for all $i$. The equation $\langle d, d \rangle =0$ together with the previous conditions give
$$0=\langle d, d \rangle = \sum_{i=1}^rp_i^2\langle \alpha_i, \alpha_i \rangle + \sum_{1 \le i < j \le r}p_ip_j\langle \alpha_j, \alpha_i \rangle.$$
Since all terms on the right hand side are non-negative, it follows that they all vanish. In particular, $\langle \alpha_i, \alpha_j \rangle = 0$ for all $1 \le i,j \le r$. This proves the second part.
\end{proof}

The following describes the rational accumulation points in case $Q$ is at most weakly hyperbolic. It follows from Theorems \ref{TheoAccuPoints} and \ref{MainTheorem}.

\begin{Cor} Let $Q$ be at most weakly hyperbolic. Then the rational accumulation points are precisely the isotropic Schur roots.
\end{Cor}

\medskip

Now, let $Q$ be any acyclic connected quiver of wild type. It is clear, using the Coxeter transformation together with Theorem \ref{TheoAccuPoints}, that there are infinitely many rational accumulation points of real Schur roots if and only if there is at least one isotropic Schur root for $Q$. Not all connected quivers of wild type have isotropic Schur roots. For irrational accumulation points of real Schur roots, we clearly have the special eigenvectors $y^-, y^+$ of the Coxeter matrix $C$. This is clear that they are irrational, because the determinant of $C$ is $\pm 1$, and we know that the special eigenvalues of $C$ are positive and different from one. If $Q$ has more than 2 vertices, we will show that there are infinitely many irrational accumulation points.  Let Acc$(Q)$ denote the set of all accumulation points of real Schur roots in $\Delta_Q$. Let Acc$_2(Q)$ denote the set of all accumulation points of the finitely generated rank two wild subcategories of $\rep(Q)$. The rest of this section is devoted to proving that Acc$(Q) = \overline{{\rm Acc}_2(Q)}$ when $Q$ is weakly hyperbolic. The tools developed work for general acyclic quivers but the proof of the main theorem only works when $Q$ is of weakly hyperbolic type.

\medskip

Recall that if $v$ is a sink or a source vertex in $Q$, then the \emph{reflection} at $v$ of $Q$ is the quiver $Q'$ obtained from $Q$ by reversing all the arrows attached to $v$.

\begin{Lemma} \label{LemmaRemoveVertex}
Let $Q$ be connected of wild type with at least three vertices. Then there exists a quiver $Q'$ which is obtained by a sequence of reflections to $Q$ having the following property. There exists a sink vertex $v$ in $Q'$ such that the quiver $Q'\backslash\{v\}$ has a connected component of tame or wild representation type.
\end{Lemma}

\begin{proof} If $Q$ is a tree, then any quiver $Q'$ whose underlying graph is the same as that of $Q$ can be obtained by a sequence of reflections applied to $Q$. Then the result follows by using the fact that there exists a vertex $v$ in $Q$ such that $Q\backslash \{v\}$ is of infinite representation type. So assume $Q$ has a (non-oriented) cycle. Take any sink vertex $v$ in $Q$. We may assume that the connected components $R_1, \ldots, R_s$ of $Q \backslash \{v\}$ are all of Dynkin type. Suppose that $s > 1$. Since $Q$ has a cycle, there exists $1 \le i \le s$ such that the quiver generated by the vertices in $R_i$ and $v$ has a cycle. Let us choose $R_j$ with $j \ne i$ and $u$ a source vertex in $R_j$. Then $u$ is also a source vertex in $Q$. Let $Q'$ be the quiver obtained by applying a reflection at $u$ to $Q$. Then $u$ becomes a sink vertex in $Q'$. Clearly, the quiver $Q'$ with the sink vertex $u$ satisfy the property of the statement. So assume $s = 1$. Consider now a vertex $u$ in $R_1$ which is either a sink or a source vertex for $R_1$. If $u$ is a source vertex, then consider $Q_u$ the quiver obtained by applying a reflection at $u$ to $Q$. If $u$ is a sink vertex, then let $Q_u$ be the quiver obtained by applying a reflection at $v$ to $Q$. In both cases $u$ is a sink vertex in $Q_u$, so we may assume that $Q_u \backslash\{u\}$ (and hence $Q\backslash\{u\}$) does not contain cycles. Since every cycle in $Q$ passes through $v$, we see that every vertex of degree one in $R_1$ has to be connected to $v$. If $R_1$ is not of type $\mathbb{A}$, then there are exactly three vertices of degree one in $R_1$. So for any such vertex $u$ in $R_1$, we see that $Q\backslash \{u\}$ contains a cycle, a contradiction. Hence, $R_1$ is of type $\mathbb{A}$. If three vertices of $R_1$ are connected to $v$, then again, if $u$ is a vertex of degree one in $R_1$, then $Q\backslash \{u\}$ contains a cycle, a contradiction. This leaves only one possibility: $Q$ is of type $\tilde{\mathbb{A}}$, and this contradicts the fact that $Q$ is of wild type.
\end{proof}

Now, we need to work with the bounded derived category $D(Q) := D^b(\rep(Q))$ of $\rep(Q)$. The reader is referred to \cite{Happel} for more details about the derived category of a finite dimensional $k$-algebra. We denote the suspension functor by $F$. We identify $\rep(Q)$ with the full subcategory of complexes concentrated in degree $0$ in $D(Q)$. If $X \in D(Q)$, then $X$ is a bounded complex
$$0 \to C_u \to C_{u+1} \to \cdots \to C_v \to 0$$ of representations in $\rep(Q)$, where $C_i$ is in degree $i$. We set $d_X = \sum(-1)^id_{C_i} \in \Z^n$. If $X,Y$ are two quasi-isomorphic complexes in $D(Q)$, then clearly, $d_X = d_Y$. Also, if $X$ is a representation, then this newly defined $d_X$ coincides with the dimension vector of $X$. It is well known that every indecomposable object in $D(Q)$ is quasi-isomorphic to an object of the form $F^iX$, where $X \in \rep(Q)$ and $i \in \Z$.

\medskip

An indecomposable object $X$ in $D^b(\rep(Q))$ is \emph{exceptional} if $\Hom_{D(Q)}(X,F^iX)=0$ whenever $i \ne 0$. A sequence $(X_1, X_2, \ldots, X_r)$ of exceptional objects in $D(Q)$ is an \emph{exceptional sequence} if $\Hom_{D(Q)}(X_i, F^sX_j)=0$ whenever $s \in \Z$ and $i < j$. Let $\tau_D$ denote the Auslander-Reiten translate in $D(Q)$. Observe that if $(X_1, X_2, \ldots, X_r)$ is an exceptional sequence in $D(Q)$, then so is $(\tau_D X_1, \tau_D X_2, \ldots, \tau_D X_r)$. If $X$ is an object in $D(Q)$, then $d_{\tau_DX} = Cd_X$, where $C$ is the Coxeter matrix. Observe also that for two objects $X,Y$ in $D(Q)$, we have $\langle d_X, d_Y \rangle = \langle Cd_X, Cd_Y \rangle$.

\medskip

There are two types of connected components of the Auslander-Reiten quiver of $D(Q)$, up to the suspension, for $Q$ connected and of infinite representation type. We have the connecting component of $D(Q)$ which contains the preprojective representations and the inverse of the suspension of the preinjective representations. This component is of shape $\Z Q^{\rm op}$. We also have the regular components of the Auslander-Reiten quiver of $\rep(Q).$

\begin{Lemma} \label{LemmaRankTwo}
Let $Q$ be any connected acyclic quiver of infinite representation type with $|Q_0| \ge 3$. Then there exists an exceptional sequence $(X,Y)$ with $X$ projective and such that $\langle d_Y, d_X \rangle \le -2$. If $Q$ is wild, then we can choose $Y$ to be non-preinjective such that $\langle d_Y, d_X \rangle \le -3$.
\end{Lemma}

\begin{proof} If $Q$ is of Euclidean type, then by using the Auslander-Reiten translate if necessary, it is sufficient to find an exceptional sequence $(X,Y)$ with $X,Y$ preprojective and such that $\langle d_Y, d_X \rangle = -2$. If $Q$ is wild, by using the Auslander-Reiten translate if necessary, it is sufficient to find an exceptional sequence $(X,Y)$ with $X$ preprojective and $Y$ non-preinjective such that $\langle d_Y, d_X \rangle \le -3$.  By working inside $D(Q)$, if $Q$ is of Euclidean type, it is sufficient to find an exceptional sequence $(X,Y)$ in the connecting component of $D(Q)$ such that $\langle d_Y, d_X \rangle = -2$. If $Q$ is wild, it is sufficient to find an exceptional sequence $(X,Y)$ in $D(Q)$ with $X$ in the connecting component and $Y$ in the connecting component or regular. Further, using this observation, the property stated is invariant under a sequence of reflection functors, because a reflection functor preserves the regular representations and the connecting component in the derived categories. Consider first the case when $Q$ is wild. By Lemma \ref{LemmaRemoveVertex}, we may assume that there exists a sink vertex $v$ such that $Q\backslash\{v\}$ has a connected component of tame or wild representation type. Let $R$ be such a connected component, and $M$ a preprojective representation in $\rep(R)$ such that for all vertex $u$ in $R$, $M(u)$ is of dimension at least $3$. Such a $M$ clearly exists by going far enough in the preprojective component of the Auslander-Reiten quiver of $\rep(R)$. Since there are infinitely many indecomposable representations in $\rep(R)$ which are successors of $M$ in the Auslander-Reiten quiver of $\rep(R)$, we see that $M$ cannot be preinjective in $\rep(Q)$. We have an exceptional sequence $(P_v, M)$ with $\langle d_M, d_{P_v} \rangle \le -3$. Finally, assume that $Q$ is Euclidean. Assume first that $Q$ is a tree. Then in this case, there exists a vertex $v$ of degree one with an arrow $u \to v$ or $v \to u$, and an exceptional representation $M$ not supported at $v$ such that $M(u)$ has dimension two; see \cite[Lemma VII 2.6]{ASS}. By using the reflection functor at $v$, we may assume that $v$ is a sink vertex. Then, we get the exceptional sequence $(P_v, M)$. We get $\langle d_M, d_{P_v} \rangle = -{\rm dim}_kM(u) = -2$, which shows the lemma in this case. So we need only to consider the case where $Q$ is of type $\widetilde{\mathbb{A}}_n$. Let $v$ be any sink vertex, and consider $M$ to be the unique, up to isomorphism, indecomposable sincere representation of $Q \backslash\{v\}$. Then $(P_v, M)$ is an exceptional sequence with the desired properties.
\end{proof}

Recall that if $(X,Y)$ is an exceptional sequence in $\rep(Q)$, then the full and thick (that is, additive, abelian and extension-closed) subcategory $\C(X,Y)$ of $\rep(Q)$ generated by $X,Y$ is equivalent to the category of representations of an acyclic quiver $Q(X,Y)$ with two vertices. It is well known that $\langle d_Y, d_X \rangle \le 0$ if and only if $X,Y$ are the non-isomorphic simple objects in $\C(X,Y)$. In this case, we shall also say that $X,Y$ are the \emph{relative-simples} in $\C(X,Y)$. Let $(X,Y)$ be an exceptional sequence in $\rep(Q)$ with $S_1, S_2$ the relative-simples.  We denote by $\ell_{X,Y}$ the line segment in $\Delta(1)$ joining the dimension vectors of $S_1, S_2$. In particular, $\ell_{X,Y}$ contains $\ch{d_X}$ and $\ch{d_Y}$. Observe that $\ell_{X,Y}$ intersects the region $q(z) < 0$ if and only if $Q(X,Y)$ is of wild type. Similarly, $\ell_{X,Y}$ intersects the region $q(z) = 0$ at a single point if and only if $Q(X,Y)$ is the of tame type (hence is the Kronecker quiver). Finally, $\ell_{X,Y}$ does not intersect the region $q(z) = 0$ if and only if $Q(X,Y)$ is either
$$\xymatrix{\bullet & \bullet} \quad \text{or} \quad \xymatrix{\bullet \ar[r] & \bullet}.$$
Observe that when $\C(X,Y)$ is of finite type, then the entire line in $\Delta(1)$ that contains $\ch{d_X},\ch{d_Y}$ does not intersect the quadric. Similarly, if $\C(X,Y)$ is of tame type, then the entire line in $\Delta(1)$ that contains $\ch{d_X},\ch{d_Y}$ intersects the region $q(z) \le 0$ at a single point. These observations indeed follow from the fact that the Tits form of a quiver of finite type (resp. tame type) is positive-definite (resp. positive semi-definite).
In case $\C(X,Y)$ is tame, the line segment $\ell_{X,Y}$ contains a unique isotropic point, denoted $y(\C(X,Y))$, which is rational.  In case $\C(X,Y)$ is of wild type, the line segment $\ell_{X,Y}$ contains exactly two isotropic points $y^-(\C(X,Y)), y^+(\C(X,Y))$ and they are both irrational. We suppose that $y^-(\C(X,Y))$ is the accumulation point of the dimension vectors of the preprojective objects in $\C(X,Y)$ while $y^+(\C(X,Y))$ is the accumulation point of the dimension vectors of the preinjective objects in $\C(X,Y)$. For convenience, in the tame case, we sometimes write $y^-(\C(X,Y)) = y^+(\C(X,Y)) = y(\C(X,Y))$. Finally, given two points $p_1, p_2$ in $\Delta(1)$, we denote by $[p_1, p_2]$ the line segment joining $p_1$ and $p_2$ in $\Delta(1)$.

\begin{Prop} \label{Prop1}Let $Q$ be any connected acyclic quiver of infinite representation type. Then there exists a sequence $((X_i,Y_i))_{i \ge 1}$ of exceptional sequences such that $(y^-(\C(X_i,Y_i)))_{i\ge 1}$ converges to $y^-$ in $\Delta_Q$. The categories $\C(X_i,Y_i)$ are of the same type (that is, tame or wild) as $Q$.
\end{Prop}

\begin{proof} Clearly, we may assume that $|Q_0| \ge 3$. By Lemma \ref{LemmaRankTwo}, there exists an exceptional sequence $(X,Y)$ with $X$ projective such that $\langle d_Y, d_X \rangle \le -2$. Moreover, if $Q$ is wild, then $Y$ is non-preinjective and $\langle d_Y, d_X \rangle \le -3$. If $Q$ is Euclidean, then $y^-(\C(X,Y)) = y^+(\C(X,Y))$ is the only isotropic point in $\Delta_Q$ and the result follows by taking $X_i = X$ and $Y_i = Y$. Assume that $Q$ is wild. Since $Y$ is not preinjective, consider $X_i = \tau^{-i}X, Y_i = \tau^{-i}Y$, which are all exceptional representations. Therefore, for $i \ge 0$, we have an exceptional sequence $(X_i, Y_i)$ with $\langle d_{Y_i}, d_{X_i} \rangle \le -3$. In particular, $X_i, Y_i$ are the relative-simples in $\C(X_i,Y_i)$. Since both $(d_{X_i})_{i \ge 0}, (d_{Y_i})_{i\ge 0}$ converge to $y^-$, by Lemma \ref{LemmaGrowthTau}, the result follows.
\end{proof}

The above proposition, with its analogue for $y^+$, ensures that the two special eigenvectors $y^-,y^+$, which are in Acc$(Q)$, are obtained as points in the closure $\overline{{\rm Acc}_2(Q)}$ of ${\rm Acc}_2(Q)$. The following result gives a nice property of an accumulation point which is not an eigenvector of the Coxeter transformation.

\begin{Prop} \label{tangent}
Let $Q$ be connected and $v \in \Delta_Q$ not an eigenvector of $C$ such that $q(v)=0$. Then the hyperplane of equation $\langle x, v \rangle = 0$ is not tangent to the quadric $q(x) = 0$.
\end{Prop}

\begin{proof}
Suppose the contrary.  Write $v = (v_1, \ldots, v_n)$. Since $Cv = -E^{-1}E^Tv \ne v$, we have $Av=(E + E^T)v \ne 0$.  Suppose that the $i$th component of $Av$ is non-zero. Consider the curve $q(x) = 0$. By setting $x = (x_1, x_2, \ldots, x_n)$, we get $\partial q/\partial x_i \ne 0$ at $v$. Hence, by the implicit function theorem, $x_i$ is a function of the other $x_j$, near $v$, through the implicit equation $q(x)=0$. Write $x_i = f(x_1, \ldots, x_{i-1}, x_{i+1}, \ldots, x_n)$, that is defined in a neighborhood of $(v_1, \ldots, v_{i-1}, v_{i+1}, \ldots, v_n)$. The condition also yields that $\partial f / \partial x_j$ exists at the point $v$ for all $j \ne i$. In particular, the tangent hyperplane at $v$ exists. The tangent hyperplane to $q(x) = \langle x, x \rangle = 0$ at $v$ is given by $x^TAv=0$.  By hypothesis, there exists $\lambda \in \mathbb{R}$ such that for all $x \in \mathbb{R}^n$, we have
$$x^TEv = \lambda x^T(E + E^T)v.$$
This gives $(1 - \lambda)Ev = \lambda E^Tv.$
Since $E$ is invertible and $v \ne 0$, $\lambda \ne 0,1$. Since $E$ is invertible and $E^{-1}E^T = -C$, we get $Cv = \frac{\lambda-1}{\lambda}v$, which is impossible.
\end{proof}

Let $X$ be an exceptional representation in $\rep(Q)$. We say that $X$ is of \emph{finite (or tame, or wild)} type if the category $X^\perp$ if of finite (resp. tame, or wild) representation type. The following says that the type coincides with that of $^\perp X$. The proof follows from the Auslander-Reiten formula and is left to the reader.

\begin{Lemma} \label{TypeAR}
Let $X$ be exceptional in $\rep(Q)$. Then $X^\perp$ and $^\perp X$ have the same representation type.
\end{Lemma}

The following lemma says that in order to study accumulation points of real Schur roots other than the special eigenvectors of $C$, we only need to consider the dimension vectors of the quasi-simple regular representations.

\begin{Lemma} \label{LemmaReg} Let $(X_i)_{i \ge 1}$ be a sequence of pairwise non-isomorphic exceptional representations such that $(d_{X_i})_{i \ge 1}$ converges to a point $p$ with $p \not \in \{y^-,y^+\}$. Then all but finitely many $X_i$ are regular quasi-simple.
\end{Lemma}

\begin{proof}
The fact that all but finitely many $X_i$ are regular follows from Lemma \ref{Lemmay+}. By \cite[Theorem 2]{HHKU}, there are finitely many $\tau$-orbits of non-quasi-simple regular exceptional representations. So if the statement of the lemma is not true, it means that there is one such $\tau$-orbit containing infinitely many of the $X_i$. By Lemma \ref{LemmaGrowthTau}, a subsequence of $(d_{X_i})_{i \ge 1}$ converges to a point in $\{y^-,y^+\}$, a contradiction.
\end{proof}

Let us turn our attention to connected wild quivers with three vertices. As we noticed, they are all of weakly hyperbolic type.

\begin{Lemma} \label{MainLemma}
Let $Q$ be a connected wild quiver with three vertices and $X$ be an exceptional wild object in $\rep(Q)$. Then there exist pairwise non-isomorphic exceptional representations $Y_1, Y_2, \ldots$ such that one of the following is satisfied.
\begin{enumerate}[$(1)$]
    \item  We have exceptional sequences $((X,Y_i))_{i \ge 1}$ with $\langle d_{Y_i}, d_X \rangle \le -3$, or
    \item We have exceptional sequences $((Y_i,X))_{i \ge 1}$ with $\langle d_{X}, d_{Y_i} \rangle \le -3$.
\end{enumerate}
\end{Lemma}

\begin{proof} Consider an exceptional sequence $(X,Y,Z)$, where we know that $\C(Y,Z)$ is wild.  Assume further that $Y,Z$ are the relative-simples in $\C(Y,Z)$. Then $$\langle d_Z, d_Y \rangle = -a \le -3.$$ Consider the two accumulation points $p_1 = y^-(\C(Y,Z)), p_2 = y^+(\C(Y,Z))$, seen as points in $\Delta(1)$. Assume that the line segment $[p_1, \ch{d_X}]$ intersects the interior of the quadric or $\langle p_1, d_X \rangle < 0$. Since $p_1$ is an accumulation point in $\C(Y,Z)$, we see that in both cases, there are infinitely many preprojective objects $U$ in $\C(Y,Z)$ such that $\langle d_U, d_X \rangle$ is negative. If these numbers are bounded below, let $b < 0$ be the minimal such number. Then there is a an almost split sequence of the form $0 \to U_1 \to E_1 \oplus \cdots \oplus E_a \to U_2 \to 0$ in the preprojective component of $\C(Y,Z)$ such that $\langle d_{U_1}, d_X \rangle = b$. Since both $U_1, E_i$ lie in $X^\perp$, we know that one of $\Hom(U_1, X), \Ext^1(U_1, X)$ is zero and one of $\Hom(E_i, X), \Ext^1(E_i, X)$ is zero. Since $\langle d_{U_1}, d_X \rangle$ is negative, $\Hom(U_1, X)=0$ and $\Ext^1(U_1, X)\ne0$. Since we have a monomorphism $U_1 \to E_i$, we get a surjective map $\Ext^1(E_i, X) \to \Ext^1(U_1, X)$, and hence $\Ext^1(E_i, X) \ne 0$ and $\Hom(E_i, X)=0$. Applying $\Hom(-,X)$ to the monomorphisms $U_1 \to E_i$, we get $\langle d_{E_i}, d_X \rangle \le \langle d_{U_1}, d_X \rangle = b$. By minimality of $b$, $\langle d_{E_i}, d_X \rangle = b$ for all $i$. Since $a > 2$, we get $\langle d_{U_2}, d_X \rangle = ab-b = (a-1)b < b$, a contradiction. Hence these numbers are not bounded and the statement follows in this case. Similarly, the statement holds if the line segment $[p_2, \ch{d_X}]$ intersects the interior of the quadric or $\langle p_2, d_X \rangle < 0$. Therefore, we may assume that none of $[\ch{d_X}, p_1]$, $[\ch{d_X}, p_2]$ intersect the interior of the quadric and $\langle p_i, d_X \rangle \ge 0$ for $i=1,2$. Now, let $(Y', Z', X)$ be an exceptional sequence where we know that $\C(Y', Z')$ is wild by Lemma \ref{TypeAR}. Assume that $Y',Z'$ are the relative-simples in $\C(Y',Z')$. Consider the two accumulations points $p_1' = y^-(\C(Y',Z')), p_2' = y^+(\C(Y',Z'))$, seen as points in $\Delta(1)$. By a similar argument as above, we may restrict to the case where none of $[p_1', \ch{d_X}]$, $[p_2', \ch{d_X}]$ intersect the interior of the quadric and $\langle d_X, p_i' \rangle \ge 0$ for $i=1,2$. Now, observe that if $X$ is not projective, then $Y'=\tau^{-}Y, Z' = \tau^-Z$ and $p_i' = \tau^-(p_i)$. Hence, if $X$ is not projective, we may also restrict to the case where none of $[p_1, \ch{\tau(d_X)}]$, $[p_2, \ch{\tau(d_X)}]$ intersect the interior of the quadric and $\langle \tau(d_X), p_i \rangle \ge 0$ for $i=1,2$.

Consider first the case where $X$ is projective. The conditions given in the first paragraph give that the region $q(x) \le 0$ of $\Delta(1)$ is entirely contained in the triangle of vertices $p_1, p_2, \ch{d_X}$. Hence, if $\alpha$ is any dimension vector lying on or inside the quadric (so $q(\alpha) \le 0$) but is not a dimension vector in $\C(Y,Z)$, then it could be written of the form $\alpha = rd_X + \beta$ where $\beta$ is strictly imaginary in $\C(Y,Z)$ and $r > 0$. In particular, $\beta$ is a Schur root. By the above conditions, we have that $(d_X, \beta)$ is a Schur sequence and $\langle \beta, d_X \rangle \ge 0$. Therefore, the canonical decomposition of $\alpha$ involves $d_X$ and hence, $\alpha$ is not a Schur root. Therefore, the only imaginary Schur roots are in $\C(Y,Z)$. This is impossible. If $X$ is not projective, then the conditions given in the first paragraph give that the region $q(x) \le 0$ of $\Delta(1)$ is entirely contained in the quadrilateral with vertices $p_1, p_2, \ch{d_X}, \ch\tau(d_X)$. Using a similar argument, one can show that a dimension vector lying in the region $q(z) \le 0$ that is not a dimension vector in $\C(Y,Z)$
has either $d_X$ or $d_{\tau(X)}$ in its canonical decomposition. This gives the same contradiction.
\end{proof}

It is well known that a connected wild quiver with at least three vertices has regular exceptional representations. By a result of Strau\ss $\,$  \cite{Strauss}, if $X$ is exceptional regular, then $X^\perp$ is connected if and only if $X$ is quasi-simple. If it is not connected, then $X^\perp$ has two connected components $\C_1,\C_2$, where $\C_1$ is equivalent to $\rep(R_1)$, where the quiver $R_1$ is acyclic connected of wild type and $\C_2$ is equivalent to $\rep(R_2)$, where the quiver $R_2$, which may be empty, is linearly oriented of Dynkin type $\mathbb{A}$. More precisely, $\C_2$ is the thick subcategory of $\rep(Q)$ generated by the quasi-simple composition factors of $X$, except its quasi-top. In particular, the rank of $R_2$ is the quasi-length of $X$ minus one. The following theorem extends the previous lemma. In particular, it says that if $X$ is exceptional regular and $Q$ is wild, then $X$ is a relative-simple of a finitely generated rank-two wild subcategory of $\rep(Q)$.

\begin{Theo} \label{TheoRankTwoSequences}
Let $Q$ be a connected quiver of wild type and $X$ be exceptional and of wild type (for instance, when $X$ is regular).  Then there exist pairwise non-isomorphic exceptional representations $Y_1, Y_2, \ldots$ such that one of the following is satisfied.
\begin{enumerate}[$(1)$]
    \item  We have exceptional sequences $((X,Y_i))_{i \ge 1}$ with $\langle d_{Y_i}, d_X \rangle \le -3$, or
    \item We have exceptional sequences $((Y_i,X))_{i \ge 1}$ with $\langle d_{X}, d_{Y_i} \rangle \le -3$.
\end{enumerate}
\end{Theo}

\begin{proof}
We proceed by induction on the rank of $\rep(Q)$. Assume $|Q_0| \ge 3$. We know that $X^\perp$ is wild. By the above observation, the category $X^\perp$ has at most two connected components $\C_1,\C_2$, where $\C_1$ is equivalent to $\rep(R_1)$, where the quiver $R_1$ is acyclic connected of wild type and $\C_2$ is equivalent to $\rep(R_2)$, where the quiver $R_2$, which may be empty, is of Dynkin type. As observed above, an indecomposable object $Z$ in $\C_2$ lies in the same Auslander-Reiten component as $X$. We will instead work in the category $\C: = ^\perp \C_2$, which is wild and connected by \cite{Strauss} and in which $X$ is a an exceptional wild object, since $X^\perp=\C_1$ in $\C$. Notice that $\C$ is the category of representations of a finite acyclic connected quiver $Q'$. So we may assume that $Q = Q'$ and $\C = \rep(Q)$. If $Q=Q'$ has three vertices, then we infer Lemma \ref{MainLemma}. Assume otherwise. Then $\C_1$ is connected of rank at least three and hence there exists an exceptional regular quasi-simple object $Y$ in $\C_1$. Then $Y$ is also regular in $\rep(Q)$. We may assume that $X$ do not lie in the wing generated by $Y$ in $\rep(Q)$, because there are infinitely many exceptional regular quasi-simple objects in $\C_1$. But then $^\perp Y$ has at most two connected components $\C_3,\C_4$, where $\C_3$ is equivalent to $\rep(R_3)$, where the quiver $R_3$ is connected of wild type and $\C_4$ is equivalent to $\rep(R_4)$, where the quiver $R_4$, which may be empty, is of Dynkin type.
By the above observation, since $X$ do not lie in the wing generated by $Y$, we see that $X \in \C_3$. Now, $\C_3$ is equivalent to $\rep(R_3)$ with $R_3$ connected wild and $X$ is an exceptional object in $\C_3$. Now, since $Y$ is regular in $\C_1$, $X^\perp \cap ^\perp Y$ is wild.  Since $^\perp Y$ is the additive hull of $\C_3$ and $\C_4$ with $\C_4$ of finite type, we see that $X^\perp \cap \C_3$ is also wild. Hence, $X$ is a wild exceptional object in $\C_3$ and we may proceed by induction.
\end{proof}

Recall that for a dimension vector $d$, we denote by $s(d)$ the sum of its entries.

\begin{Lemma} \label{Lemma5}Let $Q$ be connected wild with at least $3$ vertices. Let $(X_i)_{i \ge 1}$ be a sequence of exceptional representations such that $(d_{X_i})_{i \ge 1}$ accumulates in $\Delta_Q$ to a point different from $y^-, y^+$. Then there exists a subsequence $(X_{j})_{j \in J}$ of $(X_i)_{i \ge 1}$ that satisfies one of following.
\begin{enumerate}[$(1)$]
    \item For each $j \in J$, there exists $Y_j$ such that we have an exceptional sequence $(X_j,Y_j)$ with $\langle d_{Y_j}, d_{X_j} \rangle \le -3$.
\item For each $j \in J$, there exists $Y_j$ such that we have an exceptional sequence $(Y_j,X_j)$ with $\langle d_{X_j}, d_{Y_j} \rangle \le -3$.\end{enumerate}
\end{Lemma}

\begin{proof}
By assumption, the sequence $(d_{X_i})_{i \ge 1}$ accumulates in $\Delta_Q$ to a point $p$ which is not a special eigenvector of $C$ or of $C^{-1}$. By Lemma \ref{LemmaReg}, all but finitely many $X_i$ are regular. Combining this with Theorem \ref{TheoRankTwoSequences}, we get the wanted result.
\end{proof}

We are now ready to prove the main result of this section.

\begin{Theo} Let $Q$ be connected weakly hyperbolic. Then ${\rm Acc}(Q) = \overline{{\rm Acc}_2(Q)}$.
\end{Theo}

\begin{proof}
Let $p$ be an accumulation point of real Schur roots in $\Delta_Q$. Hence, there exists a sequence $(X_i)_{i \ge 1}$ of pairwise non-isomorphic exceptional representations such that $(d_{X_i})_{i \ge 1}$ converges to $p$ in $\Delta_Q$. Of course, we may assume that $Q$ has at least three vertices. By Proposition \ref{Prop1} and Lemma \ref{LemmaReg}, we may assume that all the $X_i$ are regular and $p \not \in \{y^-, y^+\}$. By Lemma \ref{Lemma5}, we may further assume that for $i \ge 1$, there exists an exceptional representation $Y_i$ such that $(X_i,Y_i)$ (or $(Y_i,X_i)$) is an exceptional sequence, $\C(X_i,Y_i)$ (resp. $\C(Y_i,X_i)$) is wild and $X_i,Y_i$ are the relative simples in $\C(X_i,Y_i)$ (resp. $\C(Y_i,X_i)$). Assume that we have exceptional sequences $(X_i,Y_i)$ for all $i \ge 1$. The other case where we have exceptional sequences $(Y_i,X_i)$ for all $i$ is treated in a similar way. By compactness of $\Delta_Q$, we may assume that the sequence $(d_{Y_i})_{i \ge 1}$ converges to a point $p'$ in $\Delta_Q$. So either the set $\{Y_i \mid i \ge 1\}$ is finite, up to isomorphism, or $p'$ lies on the quadric. Suppose the first case occurs. Since the sequence $(d_{X_i+Y_i})_{i \ge 1}$ lies inside the quadric and converges to $p$, this shows that $y^-(\C(X_i,Y_i))$ converges to $p$. So assume $p'$ lies on the quadric. If $p=p'$, then both sequences $(y^-(\C(X_i,Y_i)))_{i \ge 1}$, $y^+(\C(X_i,Y_i))_{i \ge 1}$ converge to $p$. So assume $p \ne p'$. Since $Q$ is weakly hyperbolic, the line segment $[p,p']$ lies in the region $q(x) \le 0$ and only $p,p'$ lie on $q(x)=0$. Since the limit of the sequences $(y^-(\C(X_i,Y_i)))_{i \ge 1}$, $y^+(\C(X_i,Y_i))_{i \ge 1}$ also lie on $[p,p']$, we see that $(y^-(\C(X_i,Y_i)))_{i \ge 1}$ converges to $p$.
\end{proof}

Observe that only the last step of the above proof uses the fact that $Q$ is weakly hyperbolic. They key fact is that for a quiver of weakly hyperbolic type, the region $q(x) = 0$ does not contain a plane. Also, note that in case $Q$ is of weakly hyperbolic type, the rational accumulation points of real Schur roots coincide with the isotropic Schur roots, and these coincide with the rational accumulation points of the subcategories of the form $\C(X,Y)$ where $(X,Y)$ is an exceptional sequence with $\langle d_Y, d_X \rangle = -2$. However, the description of accumulation points of real Schur roots in the above theorem does not use these isotropic Schur roots. Therefore, an isotropic Schur root is an accumulation point of irrational accumulation points of real Schur roots.

\medskip

Observe also that taking the closure of ${\rm Acc}_2(Q)$ in the previous theorem is essential. Indeed, none of the two special eigenvectors $y^-, y^+$ can be obtained as an accumulation point of a subcategory of the form $\C(X,Y)$ where $(X,Y)$ is an exceptional sequence. The next proposition is about this fact.

\begin{Prop}
Let $\C= \C(X,Y)$ be a rank-two finitely generated thick subcategory of $\rep(Q)$ where $Q$ is connected with $|Q_0| \ge 3$. Then $y^-, y^+$ do not belong to $\ell_{X,Y}$.
\end{Prop}

\begin{proof} We claim that $\C$ contains finitely many preprojective objects of $\rep(Q)$ and finitely many preinjective objects of $\rep(Q)$. We only prove the fist part, since the proof of the other part is similar. Assume to the contrary that $\C$ contains infinitely many preprojective objects of $\rep(Q)$. Let $X_1, X_2, \ldots$ be an infinite sequence of preprojective objects of $\rep(Q)$ such that for all $i \ge 1$, $X_i \in \C$. Then for $i \ge 1$, we have that $\tau^{j_i}X_i$ is projective for some $j_i \ge 0$. We may assume that $j_{i_2} > j_{i_1}$ if $i_2 > i_1$. Let $Z$ be such that $(X,Y,Z)$ is an exceptional sequence. If $Z$ is regular or preinjective, then for $i \ge 1$, $\tau^{j_i}Z$ is not sincere. On the other hand, as $i$ goes to infinity, $\tau^iZ$ converges to $y^-$ which is a strictly positive vector. This is a contradiction. Hence, $Z$ has to be preprojective. Now, $^\perp Z$ is equal to $(\tau^-Z)^\perp$. Let $r$ be a positive integer with $\tau^{-1+r}Z$ projective. Since all but finitely many preprojective objects in $\rep(Q)$ are sincere, we see that $(\tau^{-1+r}Z)^\perp$ contains finitely many preprojective objects. Hence, the same holds for $^\perp Z = (\tau^-Z)^\perp = (\tau^{-1+r -r}Z)^\perp$. This proves the claim.

Suppose to the contrary that $y^- \in \ell_{X,Y}$. In particular, $\C(X,Y)$ is of wild type. Now, $y^-$ is one of $y^-(\C(X,Y))$, $y^+(\C(X,Y))$. Hence, for every preprojective object $X'$ of $\C(X,Y)$ and every preinjective object $Y'$ of $\C(X,Y)$, we have that $y^-$ is a positive linear combination of $d_{X'}$ and $d_{Y'}$. By the above claim we can pick such $X',Y'$ that are regular in $\rep(Q)$. Now, if $Z$ is any indecomposable representation of $Q$ which is regular or preinjective, then from \cite[Theorem, page 240]{Pena}, $\langle y^-, d_X \rangle > 0$. This gives $\langle y^-, y^- \rangle > 0$, a contradiction. Therefore, $y^-$ does not lie in $\ell_{X,Y}$ and similarly, $y^+$ does not lie in $\ell_{X,Y}$.
\end{proof}

The following proposition holds for any connected acyclic quiver of wild type that has at least three vertices.

\begin{Prop} Let $Q$ be a connected acyclic quiver of wild type with at least three vertices. Then there are infinitely many irrational accumulation points of real Schur roots. Moreover, if there exists an isotropic Schur root, then there are infinitely many rational accumulation points of real Schur roots.
\end{Prop}

\begin{proof}
By Lemma \ref{LemmaRankTwo}, there exists an exceptional sequence $(X,Y)$ with $X$ projective and $Y$ non-preinjective such that $\langle d_Y, d_X \rangle \le -3$. Then the line $\ell_{X,Y}$ cuts the quadric at two distinct irrational points $\ell_1, \ell_2$ which are accumulation points.  Let $C$ denote the Coxeter transformation. The points in $\{C^{-i}(\ell_j) \mid j=1,2, i \ge 0\}$ are all irrational accumulation points. We claim that they are all distinct. Otherwise, one of $\ell_1, \ell_2$, say $\ell_1$, will be an eigenvector of a power of $C^{-1}$, hence an eigenvector of $C^{-1}$, since $C$ is non-singular. Now, both $\tau^{-i}X, \tau^{-i}Y$ converge to $y^-$. This gives $\ell_1 = y^-$, up to a positive scalar, a contradiction. This proves the first part of the statement. By Proposition \ref{PropAccuPoints}, any isotropic Schur root $p$ is an accumulation point of real Schur roots. So if there exists one isotropic Schur root $p$, then $\{C^i(p) \mid i \in \Z\}$ is an infinite family of isotropic Schur root and we are done. If there is no isotropic Schur root, then there is no rational accumulation point by Theorem \ref{TheoAccuPoints}.
\end{proof}

\section{Concluding remarks}

Observe that when $Q$ is not at most weakly hyperbolic, it seems harder to describe the accumulation points of real Schur roots in terms of subcategories of the form $\C(X,Y)$ where $(X,Y)$ is an exceptional sequence. For instance, we may have two isotropic Schur roots $\delta_1, \delta_2$ with $\delta_1 \perp \delta_2$ and $\delta_2 \perp \delta_1$ and we could have a rational accumulation point of real Schur roots which is a positive linear combination of $\delta_1, \delta_2$, but without being an isotropic Schur root. We do not know, however, if such a situation can occur. One should look at quivers with four vertices that are not at most weakly hyperbolic.

\medskip

We were able to prove that the canonical decomposition behaves continuously with the property of having a strictly imaginary Schur root as a summand. We believe that some continuity phenomena also occur for real Schur roots. Let $\alpha$ be a real Schur root and let $d$ be a dimension vector. Let $r_\alpha(d)$ be the multiplicity of $\alpha$ in the canonical decomposition of $d$. By Lemma \ref{NormalizedToNot}, $r_\alpha(d)s(\alpha)/s(d)$ is the coefficient of $\ch \alpha$ when writing $\ch d$ as a convex combination of the normalized Schur roots of its canonical decomposition. Let $f_\alpha(d) = r_\alpha(d)s(\alpha)/s(d)$. We conjecture that $f_\alpha$ is always continuous. It should hold, at least, when the quiver is at most weakly hyperbolic.

\bigskip

\noindent {\sc Ackowledgments:} I would like to thank NSERC and AARMS for their support while I was a postdoctoral fellow at the University of New Brunswick. This is where the project started. I would also like to thank Hugh Thomas for useful discussions and comments on the paper. Finally, I would like to thank the Department of Mathematics of the University of Connecticut for their support while I finished this paper.

\end{document}